\documentclass[12pt]{amsart}
\usepackage[english]{babel}
\usepackage[utf8x]{inputenc}
\usepackage[T1]{fontenc}
\usepackage{scribe}
\usepackage{listings}
\RequirePackage{etex}
\usepackage{tikz-cd}
\usepackage{tikz-network}
\usepackage{caption}
\usepackage{subcaption}
\usepackage{float}
\usepackage{mathtools}

\title{Lax Matrices \& Clusters for Type A \& C Q-Deformed Open Toda Chain}

\author{Corey Lunsford}

\address{C.L.: Department of Mathematics, NorthWestern University, Evanston.}
\email{coreylunsford2027@u.northwestern.edu}

\begin{document}

\maketitle

\begin{abstract}
    At the turn of the century, Etingof and Sevostyanov independently constructed a family of quantum integrable systems, quantizing the open Toda chain associated to a simple Lie group $G$. The elements of this family are parameterized by Coxeter words of the corresponding Weyl group. Twenty years later, in the works of Finkelberg, Gonin, and Tsymbaliuk, this was generalized to a family of quantum Toda chains parameterized by pairs of Coxeter words. In this paper, we show that this family is actually a single cluster integrable system written in different clusters associated to cyclic double Coxeter words. Furthermore, if we restrict the action of Hamiltonians to its positive representation, these systems become unitary equivalent.
\end{abstract}
	
\section{Introduction}

Let $G$ be a simple Lie group of rank $n$, $H\subset G$ a maximal torus, and $W=N_G(H)/H$ the corresponding Weyl group.  Let $N_\pm$ be the positive and negative maximal unipotent subgroups of $G$ and consider the open cell $G_0=N_-HN_+$. Furthermore, consider $\chi_\pm\colon N_\pm\rightarrow\mathbb{C}^*$ to be holomorphic nondegenerate characters. Then a \emph{Whittaker function} on $G_0$ with characters $\chi_\pm$ is a holomorphic function $\psi$ on $G_0$ satisfying the relation $\psi(n_-hn_+)=\chi_-(n_-)\chi_+(n_+)\psi(h)$ for any $n_\pm\in N_\pm,h\in H$. It was shown that the restriction of the Laplace operator on $G$ to the space of Whittaker functions gives the 2nd quantum Toda Hamiltonian. Thus, one gets a quantum integrable system, where the quantum integrals are restrictions to Whittaker functions of the higher Casimirs of $G$. 

Etingof \cite{E99} and Sevostyanov \cite{S99} independently applied this construction to the case when $G$ is replaced by the quantum group $U_q(\mathfrak{g})$. The key difference in this case is that $U_q(\mathfrak{n_+})$ has no non-degenerate characters. In order to deal with this issue, a choice of orientation of the Dynkin diagram (equivalently, a choice of a Coxeter word $u\in W$) can be made. Therefore, each choice leads to a $q$-deformation of the quantum Toda system. In \cite{GT19}, a natural generalization of Sevastyanov's construction to a choice of two Coxeter words is given. This leads to $3^{\text{rk}(\mathfrak{g})-1}$ quantum integrable systems which are $q$-deformations of the quantum Toda system. We will denote these as $q$-Toda systems. When $G$ is of Dynkin type $A$ \cite{FT19} and type $C$ \cite{GT19}, there is an alternative presentation by $2\times2$ Lax matrices, which is identified with the $q$-Toda systems. For each choice of a pair of Coxeter words $u,v\in W$, the corresponding $q$-Toda Hamiltonians generate a commutative subalgebra inside $D_q(H)$, the algebra of $q$-difference operators on $H$.

Since conjugation by $H$ is a Poisson map with respect to the standard Sklyanin Poisson structure, there is an induced Poisson structure on the reduced double Bruhat cell $G^{u,v}/H$. The conjugation invariant functions on $G$ form a Poisson commutative subalgebra of functions on $G$, in which there are $n$ algebraically independent such functions restricted to $G^{u,v}/H$. When $u,v\in W$ are Coxeter elements, the complex dimension of the reduced double Bruhat cell is
    \begin{equation}
        \dim(G^{u,v}/H) = l(u) + l(v) = 2n
    \end{equation}
where $l(u)$ denotes the length (i.e., number of simple reflections in the reduced expression) of $u$. Hence, the $n$ algebraically independent conjugation invariant functions determine an integrable system on $G^{u,v}/H$, called the Coxeter-Toda system. A standard choice of such functions is given by the trace in the $i$-th fundamental representation of $G$, $H_i(g) = \text{tr}(\pi_i(g))$, denoted the $i$-th \textit{Hamiltonian}. For more details, see \cite{HKKR00}.

There is a cluster realization of Coxeter-Toda systems through the language of planar directed networks, combinatorial tools first introduced by Postnikov in \cite{P06} to gain insight into totally nonnegative Grassmanians. In \cite{GSV12}, a Poisson structure was assigned to the space of edge weights of a planar directed network on a cylinder. In 
\cite{GSV11}, this Poisson structure was applied to special directed networks associated to a pair of Coxeter elements in $S_{n+1}$. A directed network in this context represents an element of $G^{u,v}/H$ when $G$ is of type $A_n$. The Poisson structure on the space of weights thus induces a Poisson structure on $G^{u,v}/H$, the phase space for a Coxeter-Toda system. Furthermore, it is shown in \cite{GSV11} that there is a cluster structure compatible with the Poisson bracket on the space of weights that assigns a quiver $Q$ to each pair $u,v\in W$. Cluster transformations are then the so-called \textit{generalized Backlund-Darboux transformations} between solutions of Coxeter-Toda systems corresponding to different Coxeter elements. A recent thesis \cite{L21} generalized this construction to all classical Dynkin types. As in \cite{FG09}, there is a canonical way to quantize a cluster structure on a Poisson variety. By choosing a polarization, one obtains a $\mathbb{C}(q)$-algebra homomorphism $\varphi_{Q}\colon\mathcal{X}^q_Q\rightarrow D_q(\mathbb{R}^n)$. Moreover, restricting $\varphi(\mathcal{X}^q_{Q})$ onto its maximal domain in $L^2(\mathbb{R}^N)$, the corresponding \textit{positive} representations (see \cite{FG09}) of different cluster charts are unitary equivalent, thus giving rise to a representation of the universally Laurent algebra $\mathbb{L}_Q^q$.

Thus, there are two ways to obtain a quantum Toda system related to a simple Lie group $G$. When $G$ is of type $A$ or $C$, the q-Toda Hamiltonians can be obtained through $2\times2$ Lax matrices. Alternatively, the cluster-Poisson structure on $G^{u,v}/H$ can be quantized to obtain a quantum cluster algebra, in which a family of a quantum Hamiltonians is a set of mutually commuting elements lying in a quantum torus algebra. In this paper, we establish an equivalence of these quantum Toda systems using the language of directed networks. Identifying $H\cong\mathbb{R}^n$, we prove the following theorem:

\begin{theorem}
    Let $G$ be a simple Lie Group of Dynkin type $A$ or $C$, $u,v$ a pair of Coxeter Weyl words, and $Q$ the quiver assigned to $G^{u,v}/H$. Let $\mathbb{L}_Q^q$ be the universally Laurent algebra associated to $Q$ as defined in \cite{FG09}. Then for each family of q-Toda Hamiltonians $\{\mathbb{H}_i\}_{1\leq i\leq n}\subset D_q(H)$ in \cite{GT19}, there is a quantum cluster chart, $\mathcal{X}_Q^q$, of $\mathbb{L}_Q^q$ and a polarization $\varphi_Q\colon\mathcal{X}^q_Q\rightarrow D_q(H)$ such that $\varphi_Q(\mathbb{H}_i)=H_i^Q$.
\end{theorem}

This theorem shows that the $3^{\text{rk}(\mathfrak{g})-1}$ families of $q$-Toda Hamiltonians for $G$ are mutation equivalent. Therefore, the following corollary is immediate:

\begin{corollary}
    The $3^{\text{rk}(\mathfrak{g})-1}$ $q$-Toda systems for $G$ restricted to their positive representation are unitary equivalent.
\end{corollary}

The results of this paper show that two different settings where the choice of a pair of Coxeter words are needed to obtain a quantum integrable system actually produce the same system. Moreover, these results can be thought of as a comparison between $2\times2$ and $N\times N$ Lax matrix descriptions. For Dynkin type A, this can be drawn from the theory of Goncharov-Kenyon integrable systems related to the rotation of a Newton polygon (\cite{GK13}, \cite{FM16}). We would like to view this paper as a first step towards an extension to other Dynkin types.

The paper is organized as follows. In Section 2 we go through the Poisson structure assigned to $G^{u,v}/H$ and the accompanying Coxeter-Toda systems from \cite{FZ99} and \cite{HKKR00}. In Section 3, we recall the Poisson structure on the space of edge weights on directed networks from \cite{L21} and offer a quantization in the sense of \cite{FG09}. In Section 4, we recall the presentation of the $3^{\text{rk}(\mathfrak{g})-1}$ families of q-Toda Hamiltonians by $2\times2$ Lax matrices found in \cite{FT19} and \cite{GT19}. Finally, we give proofs of the correspondence and provide explicit formulas for the Hamiltonians in Section 5.

\section*{Acknowledgements}

I give my gratitude to Alexander Shapiro for his guidance, support and fruitful discussions at all stages of this project. I would also like to thank the University of Edinburgh for their hospitality and support to help complete this paper. This work was partially supported by the Royal Society University Research Fellowship, grant No. URF-R1-201530.

\section{Double Bruhat Cells \& Coxeter-Toda Systems}

In this section, we recall the construction of reduced Double Bruhat Cells and the accompanying symplectic structure from \cite{FZ99} and \cite{HKKR00}. Let $G$ be a simple complex Lie group of rank $n$. Furthermore, let $B_+,B_-$ ($U_+,U_-$) be a choice of positive and negative Borel (unipotent) subgroups and $H=B_+\cap B_-$ be the maximal torus of $G$. Recall that the Lie algebra, $\mathfrak{g}$, of $G$ has a decomposition $\mathfrak{g}=\mathfrak{n}_-\oplus\mathfrak{h}\oplus\mathfrak{n}_+$ where $\mathfrak{n}_-,\mathfrak{h},\mathfrak{n}_+$ are the Lie algebras of $U_-,H,U_+$, respectively. We can fix a basis $\alpha_i\in\mathfrak{h}$ of simple roots and a dual basis of simple coroots $\alpha_i^\vee\in\mathfrak{h}^*$ for $i\in[1,n]$ such that $\alpha_j(\alpha_i^\vee)=C_{ij}$ where $C$ is the Cartan matrix. This allows us to fix Chevalley generators $e_i\in\mathfrak{n}_+$ and $e_{-i}\in\mathfrak{n}_-$. They give rise to the one-parameter subgroups $E_i(t),E_{-i}(t)\in G$ for $t\in\mathbb{C}^\times$.

The group $G$ admits two Bruhat decompositions given by
    \begin{equation}
        G = \bigsqcup_{u\in W} B_+\dot{u}B_+ = \bigsqcup_{v\in W} B_-\dot{v}B_-
    \end{equation}
where $\dot{u},\dot{v}$ are representatives of the Weyl group $W=N(H)/H$ in $G$. The \textit{double Bruhat cell} of $G$ with respect to $u,v\in W$ is denoted
    \begin{equation}
        G^{u,v} = B_+\dot{u}B_+ \cap B_-\dot{v}B_-,
    \end{equation}
which allows us to decompose $G$ in the following way:
    \begin{equation}
        G = \bigsqcup_{u,v\in W} G^{u,v}.
    \end{equation}

Let $(s_i)_{i\in [1,n]}$ be the simple transpositions for $W$. Then $u\in W$ can be written as $u=s_{i_1}\cdots s_{i_m}$ for some $i_1,\ldots,i_m\in [1,n]$. A \textit{word} corresponding to $u$ is defined as the sequence $\mathbf{i}=(i_1,\ldots,i_m)$ for $i_j\in[1,n]$. Let $l(u)$ be the \textit{length} of $u\in W$, i.e. the number of simple reflections in the decomposition of $u$. Then a word is reduced if $l(u)$ is minimal. Furthermore, a \textit{double reduced word} for $u,v\in W$ is a tuple $\mathbf{i}$ such that the entries are a shuffling of the letters of $-\mathbf{i}_u$ and $\mathbf{i}_v$, where $\mathbf{i}_u,\mathbf{i}_v$ are  reduced words for $u$ and $v$, respectively. A double reduced word is called \textit{unmixed} if it can be written as $\mathbf{i}=(-\mathbf{i}_u,\mathbf{i}_v)$.

In \cite{FZ99}, it was proved that if $\mathbf{i}=(i_1,\ldots,i_m)$ is a double reduced word for $(u,v)\in W\times W$, then the map $H\times\mathbb{C}^m\rightarrow G$ such that
    \begin{equation}
        (h,a_1,\ldots,a_m)\mapsto hE_{i_1}(a_1)\cdots E_{i_m}(a_m)
    \end{equation}
restricts to a biregular isomorphism on a Zariski open dense subset of $G^{u,v}$. Effectively, this allows us to decompose any element $g\in G^{u,v}$ as
    \begin{equation}
        g = D(t_1,\ldots,t_n)E_{i_1}(a_1)\cdots E_{i_m}(a_m)
    \end{equation}
for $t_j,a_k\in\mathbb{C}^\times$, where $D(t_1,\ldots,t_n)=\prod_{i=1}^n t_i^{\alpha_i^\vee}$.

Since conjugation by the Cartan subgroup $H$ preserves $G^{u,v}$, it is possible to define the quotient $G^{u,v}/H$. If $u,v$ are Coxeter elements and $\mathbf{i}$ unmixed, then in the notation of \cite{L21}, any element $\bar{g}\in G^{u,v}/H$ can be factorized as
    \begin{equation}
        \bar{g} = E_{i_1}(1)\cdots E_{i_n}(1)D(t_1,\ldots,t_n)E_{i_{n+1}}(c_{i_{n+1}})\cdots E_{i_{2n}}(c_{i_{2n}})
    \end{equation}
for $t_j,c_k\in\mathbb{C}^\times$.

\section{Cluster Structure of q-Toda Systems}

\subsection{Directed Networks}

In this section, we will construct special directed networks on the disk corresponding to an element of $G^{u,v}$. The Type A case was done in \cite{FZ99} and \cite{GSV11}, and this was generalized to all classical types in \cite{Y09}. We will recall this construction for types A and C.

Given a double reduced Coxeter word $\mathbf{i}=(i_1,\ldots,i_m)$ for $(u,v)\in W\times W$, a directed network $N_{u,v}(\mathbf{i})$ can be formed out of \emph{elementary chips} to represent the factorization scheme on $G^{u,v}/H$ given by
    \begin{equation}
        g = E_{i_1}(1)\cdots E_{i_n}(1)D(t_1,\ldots,t_n)E_{i_{r+1}}\cdots E_{i_{2n}}(c_{i_{2n}}).
    \end{equation}
The elementary chips defined explicitly below are glued together from left to right in the order specified by the above factorization scheme. Weights are assigned to edges according to the elementary chips (See Figures 1 \& 2). The matrix element $g_{ij}$ is then given by the sum over all path weights from source $i$ to sink $j$. Thus, the directed networks give a clear combinatorial description of the matrix elements of $g$, discussed further in Section 6.4.

\subsubsection{Type A}

Take $G=SL_{n+1}(\mathbb{C})$. A set of Chevalley generators for $\mathfrak{g}=\mathfrak{sl}_2$ are given by $e_i=e_{i,i+1},e_{-i}=e_{i+1,i}$, and $h_i=e_{i,i}-e_{i+1,i+1}$ for $1\leq i\leq n$, where $e_{ij}$ is the matrix with 1 in the $(i,j)$-th entry and 0's everywhere else. This gives us the group generators
    \begin{equation}
    \begin{split}
        E_i(a_i) & = I_{n+1} + a_ie_{i,i+1} \\
        E_{-i}(b_i) & = I_{n+1} + b_ie_{i+1,i} \\
        D(t_1,\ldots,t_n) & = \text{diag}(t_1,t_1^{-1}t_2,\ldots,t_{n-1}^{-1}t_n,t_n^{-1})
    \end{split}
    \end{equation}
where $I_{n+1}$ is the $(n+1)\times(n+1)$ identity matrix. These matrices are represented by the elementary chips in \textbf{Figure 1}. Explicitly, we can see that the $(i,j)$-th matrix element is given by the weight of the path from row $i$ to row $j$.

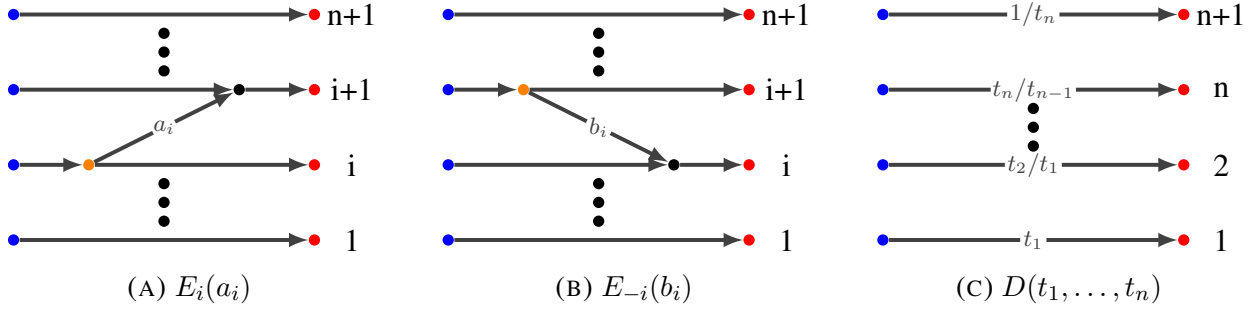
\begin{figure}[H]
    \centering
    \begin{subfigure}[b]{0.3\textwidth}
        \centering
        \begin{tikzpicture}
            \tikzset{VertexStyle/.style = {shape = circle,fill = black,minimum size = 1.5mm,inner sep=0pt}}
            \Vertex[color=blue]{A1} \Vertex[x=4,color=red]{C1}
            \Edge[Direct](A1)(C1)
            \Vertex[x=2,y=0.25,color=black]{X1}
            \Vertex[x=2,y=0.5,color=black]{X2}
            \Vertex[x=2,y=0.75,color=black]{X3}
            \Vertex[y=1,color=blue]{A2} \Vertex[x=1,y=1,color=orange]{B2} \Vertex[x=4,y=1,color=red]{C2} 
            \Edge[Direct](A2)(B2) \Edge[Direct](B2)(C2)
            \Vertex[y=2,color=blue]{A3} \Vertex[x=3,y=2,color=black]{B3} \Vertex[x=4,y=2,color=red]{C3} 
            \Edge[Direct](A3)(B3) \Edge[Direct](B3)(C3) \Edge[Direct,label=$a_i$](B2)(B3)
            \Vertex[x=2,y=2.25,color=black]{X4}
            \Vertex[x=2,y=2.5,color=black]{X5}
            \Vertex[x=2,y=2.75,color=black]{X6}
            \Vertex[y=3,color=blue]{A4} \Vertex[x=4,y=3,color=red]{C4}
            \Edge[Direct](A4)(C4)
            \Text[x=4.5]{1}
            \Text[x=4.5,y=1]{i}
            \Text[x=4.5,y=2]{i+1}
            \Text[x=4.5,y=3]{n+1}
        \end{tikzpicture}
    \caption{$E_i(a_i)$}
    \label{fig:1a}
    \end{subfigure}
    \hfill
    \begin{subfigure}[b]{0.3\textwidth}
        \centering
        \begin{tikzpicture}
            \tikzset{VertexStyle/.style = {shape = circle,fill = black,minimum size = 1.5mm,inner sep=0pt}}
            \Vertex[color=blue]{A1} \Vertex[x=4,color=red]{C1}
            \Edge[Direct](A1)(C1)
            \Vertex[x=2,y=0.25,color=black]{X1}
            \Vertex[x=2,y=0.5,color=black]{X2}
            \Vertex[x=2,y=0.75,color=black]{X3}
            \Vertex[y=1,color=blue]{A2} \Vertex[x=3,y=1,color=black]{B2} \Vertex[x=4,y=1,color=red]{C2} 
            \Edge[Direct](A2)(B2) \Edge[Direct](B2)(C2)
            \Vertex[y=2,color=blue]{A3} \Vertex[x=1,y=2,color=orange]{B3} \Vertex[x=4,y=2,color=red]{C3} 
            \Edge[Direct](A3)(B3) \Edge[Direct](B3)(C3) \Edge[Direct,label=$b_i$](B3)(B2)
            \Vertex[x=2,y=2.25,color=black]{X4}
            \Vertex[x=2,y=2.5,color=black]{X5}
            \Vertex[x=2,y=2.75,color=black]{X6}
            \Vertex[y=3,color=blue]{A4} \Vertex[x=4,y=3,color=red]{C4}
            \Edge[Direct](A4)(C4)
            \Text[x=4.5]{1}
            \Text[x=4.5,y=1]{i}
            \Text[x=4.5,y=2]{i+1}
            \Text[x=4.5,y=3]{n+1}
        \end{tikzpicture}
    \caption{$E_{-i}(b_i)$}
    \label{fig:1b}
    \end{subfigure}
    \hfill
    \begin{subfigure}[b]{0.3\textwidth}
        \centering
        \begin{tikzpicture}
            \tikzset{VertexStyle/.style = {shape = circle,fill = black,minimum size = 1.5mm,inner sep=0pt}}
            \Vertex[color=blue]{A1} \Vertex[x=4,color=red]{C1}
            \Edge[Direct,label=$t_1$](A1)(C1)
            \Vertex[y=1,color=blue]{A2} \Vertex[x=4,y=1,color=red]{C2}
            \Edge[Direct,label=$t_2/t_1$](A2)(C2)
            \Vertex[x=2,y=1.25,color=black]{X1}
            \Vertex[x=2,y=1.5,color=black]{X2}
            \Vertex[x=2,y=1.75,color=black]{X3}
            \Vertex[y=2,color=blue]{A3} \Vertex[x=4,y=2,color=red]{C3}
            \Edge[Direct,label=$t_n/t_{n-1}$](A3)(C3)
            \Vertex[y=3,color=blue]{A4} \Vertex[x=4,y=3,color=red]{C4}
            \Edge[Direct,label=$1/t_n$](A4)(C4)
            \Text[x=4.5]{1} \Text[x=4.5,y=1]{2} \Text[x=4.5,y=2]{n} \Text[x=4.5,y=3]{n+1}
        \end{tikzpicture}
    \caption{$D(t_1,\ldots,t_n)$}
    \label{fig:1c}
    \end{subfigure}
    \caption{Type $A_n$ Elementary Chips}
    \label{fig:my_label}
\end{figure}

As an example, the type $A_3$ network diagram $N_{u,v}(\mathbf{i})$ for the standard double Coxeter word $\mathbf{i_0}=(-1,\ldots -n,1,\ldots,n)$ is given in \textbf{Figure 2}.

\begin{figure}
    \centering
    \begin{tikzpicture}
        \tikzset{VertexStyle/.style = {shape = circle,fill = black,minimum size = 1.5mm,inner sep=0pt}}
        \Vertex[color=blue]{A1} \Vertex[x=2,color=black]{B1} \Vertex[x=7,color=orange]{C1} \Vertex[x=13,color=red]{D1}
        \Edge[Direct](A1)(B1) \Edge[Direct,label=$t_1$](B1)(C1) \Edge[Direct](C1)(D1)
        \Vertex[y=1,color=blue]{A2} \Vertex[x=1,y=1,color=orange]{B2} \Vertex[x=4,y=1,color=black]{C2} \Vertex[x=8,y=1,color=black]{D2} \Vertex[x=9,y=1,color=orange]{E2} \Vertex[x=13,y=1,color=red]{F2}
        \Edge[Direct](A2)(B2) \Edge[Direct](B2)(C2) \Edge[Direct,label=$t_1^{-1}t_2$](C2)(D2) \Edge[Direct](D2)(E2) \Edge[Direct](E2)(F2) \Edge[Direct](B2)(B1) \Edge[Direct,label=$c_1$](C1)(D2)
        \Vertex[y=2,color=blue]{A3} \Vertex[x=3,y=2,color=orange]{B3} \Vertex[x=6,y=2,color=black]{C3} \Vertex[x=10,y=2,color=black]{D3} \Vertex[x=11,y=2,color=orange]{E3} \Vertex[x=13,y=2,color=red]{F3}
        \Edge[Direct](A3)(B3) \Edge[Direct](B3)(C3) \Edge[Direct,label=$t_2^{-1}t_3$](C3)(D3) \Edge[Direct](D3)(E3) \Edge[Direct](E3)(F3) \Edge[Direct](B3)(C2) \Edge[Direct,label=$c_2$](E2)(D3)
        \Vertex[y=3,color=blue]{A4} \Vertex[x=5,y=3,color=orange]{B4} \Vertex[x=12,y=3,color=black]{C4} \Vertex[x=13,y=3,color=red]{D4}
        \Edge[Direct](A4)(B4) \Edge[Direct,label=$t_3^{-1}$](B4)(C4) \Edge[Direct](C4)(D4) \Edge[Direct](B4)(C3) \Edge[Direct,label=$c_3$](E3)(C4)
        \Text[x=13.5]{1} \Text[x=13.5,y=1]{2} \Text[x=13.5,y=2]{3} \Text[x=13.5,y=3]{4}
    \end{tikzpicture}
    \caption{Type $A_3$ Directed Network for $\mathbf{i_0}$}
    \label{fig:2}
\end{figure}
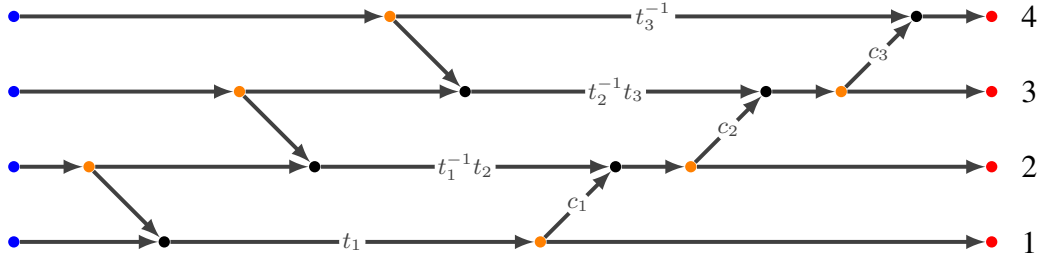

\subsubsection{Type C}

Now take $G=Sp_{2n}(\mathbb{C})$. The Chevalley generators for the Lie algebra $\mathfrak{sp}_{2n}$ of type $C_n$ are given by $e_i = e_{i,i+1} + e_{2n-i,2n+1-i}$, $e_{-i} = e_{i+1,i} + e_{2n+1-i,2n-i}$ for $1\leq i\leq n-1$, $e_n = e_{n,n+1}$, $e_{-n} = e_{n+1,n}$, and $h_i = e_{i,i} - e_{i+1,i+1} + e_{2n-i,2n-i} - e_{2n+1-i,2n+1-i}$ for $1\leq i\leq n-1$ and $h_n = e_{n,n} - e_{n+1,n+1}$. The corresponding Lie group elements are therefore
    \begin{equation}
    \begin{split}
        E_i(a_i) & = I_{2n} + a_ie_{i,i+1} + a_ie_{2n-i,2n+1-i}, \quad E_n = I_{2n} + a_ne_{n,n+1} \\
        E_{-i}(b_i) & = I_{2n} + b_ie_{i+1,i} + b_ie_{2n+1-i,2n-i}, \quad E_{-n} = I_{2n} + b_ne_{n+1,n} \\
        D(t_1,\ldots,t_n) & = \text{diag}(t_1,t_1^{-1}t_2,\ldots,t_{n-1}^{-1}t_n,t_{n-1}t_n^{-1},\ldots,t_1t_2^{-1},t_1^{-1}).
    \end{split}
    \end{equation}
The corresponding elementary chips are given in \textbf{Figure 3}.

\begin{figure}[H]
    \centering
    \begin{subfigure}[b]{0.3\textwidth}
        \centering
        \begin{tikzpicture}
            \tikzset{VertexStyle/.style = {shape = circle,fill = black,minimum size = 1.5mm,inner sep=0pt}}
            \Vertex[color=blue]{A1} \Vertex[x=4,color=red]{C1}
            \Edge[Direct](A1)(C1)
            \Vertex[x=2,y=0.25,color=black]{X1}
            \Vertex[x=2,y=0.5,color=black]{X2}
            \Vertex[x=2,y=0.75,color=black]{X3}
            \Vertex[y=1,color=blue]{A2} \Vertex[x=1,y=1,color=orange]{B2} \Vertex[x=4,y=1,color=red]{C2}
            \Edge[Direct](A2)(B2) \Edge[Direct](B2)(C2)
            \Vertex[y=2,color=blue]{A3} \Vertex[x=3,y=2,color=black]{B3} \Vertex[x=4,y=2,color=red]{C3}
            \Edge[Direct](A3)(B3) \Edge[Direct](B3)(C3) \Edge[Direct,label=$a_i$](B2)(B3)
            \Vertex[x=2,y=2.25,color=black]{X4}
            \Vertex[x=2,y=2.5,color=black]{X5}
            \Vertex[x=2,y=2.75,color=black]{X6}
            \Vertex[y=3,color=blue]{A4} \Vertex[x=1,y=3,color=orange]{B4} \Vertex[x=4,y=3,color=red]{C4}
            \Edge[Direct](A4)(B4) \Edge[Direct](B4)(C4)
            \Vertex[y=4,color=blue]{A5} \Vertex[x=3,y=4,color=black]{B5} \Vertex[x=4,y=4,color=red]{C5}
            \Edge[Direct](A5)(B5) \Edge[Direct](B5)(C5) \Edge[Direct,label=$a_i$](B4)(B5)
            \Vertex[x=2,y=4.25,color=black]{X7}
            \Vertex[x=2,y=4.5,color=black]{X8}
            \Vertex[x=2,y=4.75,color=black]{X9}
            \Vertex[y=5,color=blue]{A6} \Vertex[x=4,y=5,color=red]{C6}
            \Edge[Direct](A6)(C6)
            \Text[x=4.75]{1} \Text[x=4.75,y=1]{i} \Text[x=4.75,y=2]{i+1} \Text[x=4.75,y=3]{2n-i} \Text[x=4.75,y=4]{2n+1-i} \Text[x=4.75,y=5]{2n}
        \end{tikzpicture}
        \caption{$E_i(a_i)$}
        \label{fig:3a}
    \end{subfigure}
    \hspace{15mm}
    \begin{subfigure}[b]{0.3\textwidth}
        \centering
        \begin{tikzpicture}
            \tikzset{VertexStyle/.style = {shape = circle,fill = black,minimum size = 1.5mm,inner sep=0pt}}
            \Vertex[color=blue]{A1} \Vertex[x=4,color=red]{C1}
            \Edge[Direct](A1)(C1)
            \Vertex[x=2,y=0.25,color=black]{X1}
            \Vertex[x=2,y=0.5,color=black]{X2}
            \Vertex[x=2,y=0.75,color=black]{X3}
            \Vertex[y=1,color=blue]{A2} \Vertex[x=3,y=1,color=black]{B2} \Vertex[x=4,y=1,color=red]{C2}
            \Edge[Direct](A2)(B2) \Edge[Direct](B2)(C2)
            \Vertex[y=2,color=blue]{A3} \Vertex[x=1,y=2,color=orange]{B3} \Vertex[x=4,y=2,color=red]{C3}
            \Edge[Direct](A3)(B3) \Edge[Direct](B3)(C3) \Edge[Direct,label=$b_i$](B3)(B2)
            \Vertex[x=2,y=2.25,color=black]{X4}
            \Vertex[x=2,y=2.5,color=black]{X5}
            \Vertex[x=2,y=2.75,color=black]{X6}
            \Vertex[y=3,color=blue]{A4} \Vertex[x=3,y=3,color=black]{B4} \Vertex[x=4,y=3,color=red]{C4}
            \Edge[Direct](A4)(B4) \Edge[Direct](B4)(C4)
            \Vertex[y=4,color=blue]{A5} \Vertex[x=1,y=4,color=orange]{B5} \Vertex[x=4,y=4,color=red]{C5}
            \Edge[Direct](A5)(B5) \Edge[Direct](B5)(C5) \Edge[Direct,label=$b_i$](B5)(B4)
            \Vertex[x=2,y=4.25,color=black]{X7}
            \Vertex[x=2,y=4.5,color=black]{X8}
            \Vertex[x=2,y=4.75,color=black]{X9}
            \Vertex[y=5,color=blue]{A6} \Vertex[x=4,y=5,color=red]{C6}
            \Edge[Direct](A6)(C6)
            \Text[x=4.75]{1} \Text[x=4.75,y=1]{i} \Text[x=4.75,y=2]{i+1} \Text[x=4.75,y=3]{2n-i} \Text[x=4.75,y=4]{2n+1-i} \Text[x=4.75,y=5]{2n}
        \end{tikzpicture}
        \caption{$E_{-i}(b_i)$}
        \label{fig:3b}
    \end{subfigure}
    \par\bigskip
    \begin{subfigure}[b]{0.3\textwidth}
        \centering
        \begin{tikzpicture}
            \tikzset{VertexStyle/.style = {shape = circle,fill = black,minimum size = 1.5mm,inner sep=0pt}}
            \Vertex[color=blue]{A1} \Vertex[x=4,color=red]{C1}
            \Edge[Direct](A1)(C1)
            \Vertex[x=2,y=0.25,color=black]{X1}
            \Vertex[x=2,y=0.5,color=black]{X2}
            \Vertex[x=2,y=0.75,color=black]{X3}
            \Vertex[y=1,color=blue]{A2} \Vertex[x=1,y=1,color=orange]{B2} \Vertex[x=4,y=1,color=red]{C2} 
            \Edge[Direct](A2)(B2) \Edge[Direct](B2)(C2)
            \Vertex[y=2,color=blue]{A3} \Vertex[x=3,y=2,color=black]{B3} \Vertex[x=4,y=2,color=red]{C3} 
            \Edge[Direct](A3)(B3) \Edge[Direct](B3)(C3) \Edge[Direct,label=$a_n$](B2)(B3)
            \Vertex[x=2,y=2.25,color=black]{X4}
            \Vertex[x=2,y=2.5,color=black]{X5}
            \Vertex[x=2,y=2.75,color=black]{X6}
            \Vertex[y=3,color=blue]{A4} \Vertex[x=4,y=3,color=red]{C4}
            \Edge[Direct](A4)(C4)
            \Text[x=4.5]{1} \Text[x=4.5,y=1]{n} \Text[x=4.5,y=2]{n+1} \Text[x=4.5,y=3]{2n}
        \end{tikzpicture}
        \caption{$E_n(a_n)$}
        \label{fig:3c}
    \end{subfigure}
    \hfill
    \begin{subfigure}[b]{0.3\textwidth}
        \centering
        \begin{tikzpicture}
            \tikzset{VertexStyle/.style = {shape = circle,fill = black,minimum size = 1.5mm,inner sep=0pt}}
            \Vertex[color=blue]{A1} \Vertex[x=4,color=red]{C1}
            \Edge[Direct](A1)(C1)
            \Vertex[x=2,y=0.25,color=black]{X1}
            \Vertex[x=2,y=0.5,color=black]{X2}
            \Vertex[x=2,y=0.75,color=black]{X3}
            \Vertex[y=1,color=blue]{A2} \Vertex[x=3,y=1,color=black]{B2} \Vertex[x=4,y=1,color=red]{C2} 
            \Edge[Direct](A2)(B2) \Edge[Direct](B2)(C2)
            \Vertex[y=2,color=blue]{A3} \Vertex[x=1,y=2,color=orange]{B3} \Vertex[x=4,y=2,color=red]{C3} 
            \Edge[Direct](A3)(B3) \Edge[Direct](B3)(C3) \Edge[Direct,label=$b_n$](B3)(B2)
            \Vertex[x=2,y=2.25,color=black]{X4}
            \Vertex[x=2,y=2.5,color=black]{X5}
            \Vertex[x=2,y=2.75,color=black]{X6}
            \Vertex[y=3,color=blue]{A4} \Vertex[x=4,y=3,color=red]{C4}
            \Edge[Direct](A4)(C4)
            \Text[x=4.5]{1} \Text[x=4.5,y=1]{n} \Text[x=4.5,y=2]{n+1} \Text[x=4.5,y=3]{2n}
        \end{tikzpicture}
        \caption{$E_{-n}(b_n)$}
        \label{fig:3d}
    \end{subfigure}
    \hfill
    \begin{subfigure}[b]{0.3\textwidth}
        \centering
        \begin{tikzpicture}
            \tikzset{VertexStyle/.style = {shape = circle,fill = black,minimum size = 1.5mm,inner sep=0pt}}
            \Vertex[color=blue]{A1} \Vertex[x=4,color=red]{C1}
            \Edge[Direct,label=$t_1$](A1)(C1)
            \Vertex[x=2,y=0.25,color=black]{X1}
            \Vertex[x=2,y=0.5,color=black]{X2}
            \Vertex[x=2,y=0.75,color=black]{X3}
            \Vertex[y=1,color=blue]{A2} \Vertex[x=4,y=1,color=red]{C2} 
            \Edge[Direct,label=$t_n/t_{n-1}$](A2)(C2)
            \Vertex[y=2,color=blue]{A3} \Vertex[x=4,y=2,color=red]{C3} 
            \Edge[Direct,label=$t_{n-1}/t_n$](A3)(C3)
            \Vertex[x=2,y=2.25,color=black]{X4}
            \Vertex[x=2,y=2.5,color=black]{X5}
            \Vertex[x=2,y=2.75,color=black]{X6}
            \Vertex[y=3,color=blue]{A4} \Vertex[x=4,y=3,color=red]{C4}
            \Edge[Direct,label=$1/t_1$](A4)(C4)
            \Text[x=4.5]{1} \Text[x=4.5,y=1]{n} \Text[x=4.5,y=2]{n+1} \Text[x=4.5,y=3]{2n}
        \end{tikzpicture}
        \caption{$D(t_1,\ldots,t_n)$}
        \label{fig:3e}
    \end{subfigure}
    \caption{Type $C_n$ Elementary Chips}
    \label{fig:3}
\end{figure}
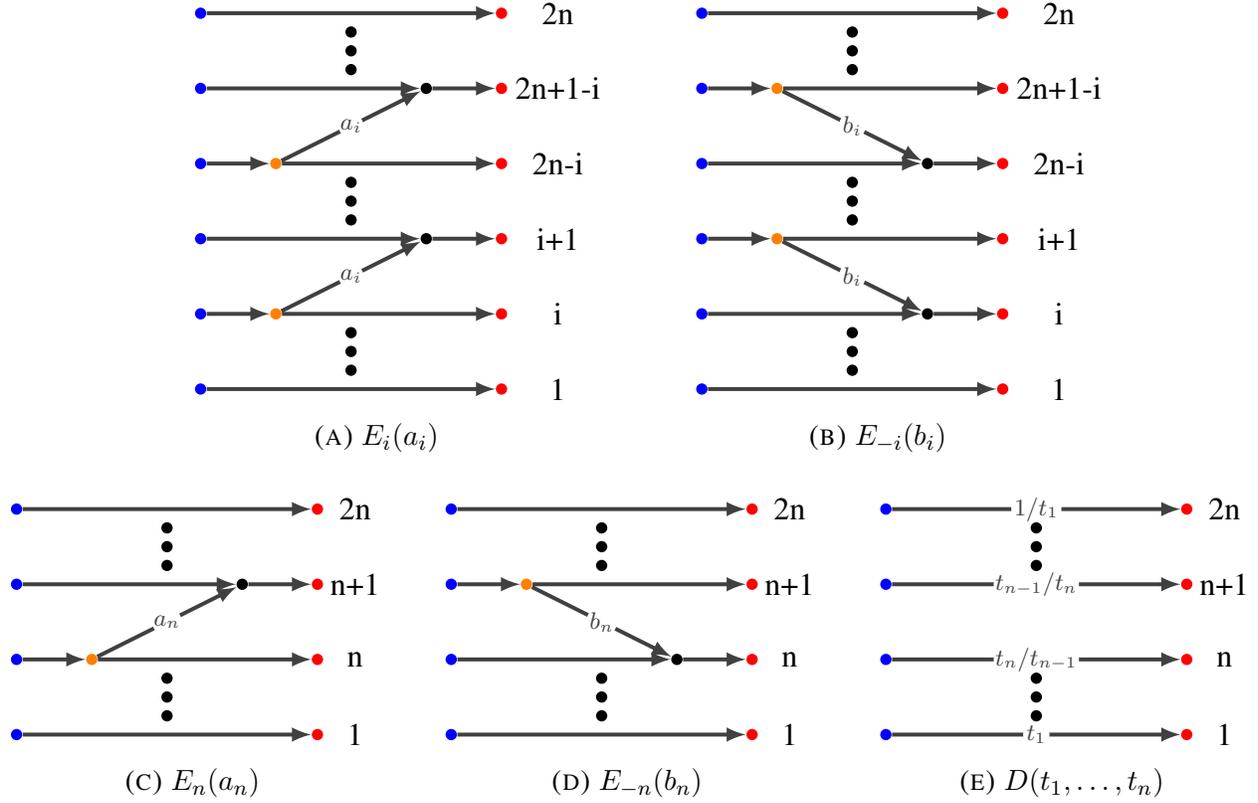

The Type $C_2$ network diagram $N_{u,v}(\mathbf{i})$ for the standard double Coxeter word $\mathbf{i_0}=(-1,\ldots -n,1,\ldots,n)$ is given in \textbf{Figure 4}.

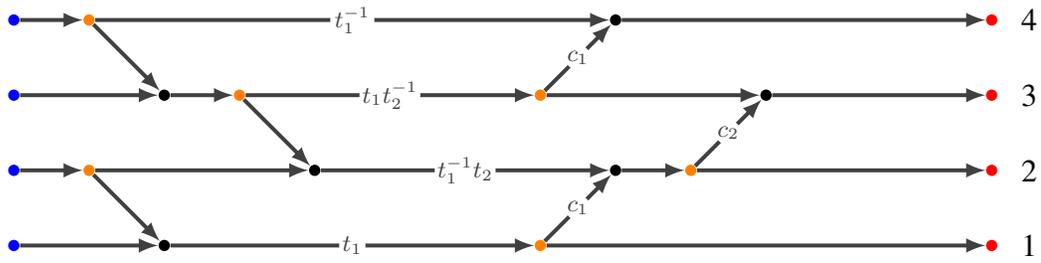
\begin{figure}[H]
    \centering
    \begin{tikzpicture}
        \tikzset{VertexStyle/.style = {shape = circle,fill = black,minimum size = 1.5mm,inner sep=0pt}}
        \Vertex[color=blue]{A1} \Vertex[x=2,color=black]{B1} \Vertex[x=7,color=orange]{C1} \Vertex[x=13,color=red]{D1}
        \Edge[Direct](A1)(B1) \Edge[Direct,label=$t_1$](B1)(C1) \Edge[Direct](C1)(D1)
        \Vertex[y=1,color=blue]{A2} \Vertex[x=1,y=1,color=orange]{B2} \Vertex[x=4,y=1,color=black]{C2} \Vertex[x=8,y=1,color=black]{D2} \Vertex[x=9,y=1,color=orange]{E2} \Vertex[x=13,y=1,color=red]{F2}
        \Edge[Direct](A2)(B2) \Edge[Direct](B2)(C2) \Edge[Direct,label=$t_1^{-1}t_2$](C2)(D2) \Edge[Direct](D2)(E2) \Edge[Direct](E2)(F2) \Edge[Direct](B2)(B1) \Edge[Direct,label=$c_1$](C1)(D2)
        \Vertex[y=2,color=blue]{A3} \Vertex[x=2,y=2,color=black]{B3} \Vertex[x=3,y=2,color=orange]{C3} \Vertex[x=7,y=2,color=orange]{D3} \Vertex[x=10,y=2,color=black]{E3} \Vertex[x=13,y=2,color=red]{F3}
        \Edge[Direct](A3)(B3) \Edge[Direct](B3)(C3) \Edge[Direct,label=$t_1t_2^{-1}$](C3)(D3) \Edge[Direct](D3)(E3) \Edge[Direct](E3)(F3) \Edge[Direct](C3)(C2) \Edge[Direct,label=$c_2$](E2)(E3)
        \Vertex[y=3,color=blue]{A4} \Vertex[x=1,y=3,color=orange]{B4} \Vertex[x=8,y=3,color=black]{C4} \Vertex[x=13,y=3,color=red]{D4}
        \Edge[Direct](A4)(B4) \Edge[Direct,label=$t_1^{-1}$](B4)(C4) \Edge[Direct](C4)(D4) \Edge[Direct](B4)(B3) \Edge[Direct,label=$c_1$](D3)(C4)
        \Text[x=13.5]{1} \Text[x=13.5,y=1]{2} \Text[x=13.5,y=2]{3} \Text[x=13.5,y=3]{4}
    \end{tikzpicture}
    \caption{Type $C_2$ Directed Network for $\mathbf{i_0}$}
    \label{fig:4}
\end{figure}

\subsection{Cluster Varieties}

In this section, we will recall some basic facts about cluster varieties and quantum cluster varieties following \cite{FG09} and \cite{FG06}.

\subsubsection{Classical Cluster Varieties}

\begin{definition}
    A \textit{seed} is the datum $\Sigma=(\Lambda,(*,*),\{e_i\},\{d_i\})$, where
    \begin{enumerate}
        \item $\Lambda$ is a lattice;
        \item $(*,*)$ is a skew-symmetric $\mathbb{Q}$-valued bilinear form on $\Lambda$;
        \item $\{e_i\}$ is a basis of the lattice $\Lambda$, and $I_0$ is a subset called the frozen basis vectors;
        \item $\{d_i\}$ are positive integers assigned to the basis vectors such that
            \begin{equation}
                \epsilon_{ij} \coloneqq (e_i,e_j)d_j \in \mathbb{Z}, \quad \text{unless } i,j\in I_0\times I_0.
            \end{equation}
        The matrix $\epsilon_{ij}$ is denoted the \textit{exchange matrix} for $\Sigma$.
    \end{enumerate}
\end{definition}

One can perform a \textit{cluster mutation} of a seed at an index $k$ to get a new seed, denoted $\mu_k(\Sigma)$. The new quadruple for $\mu_k(\Sigma)$ is $(I,I_0,\mu_k(B),D)$ where
    \begin{equation}
        \mu_k(\epsilon)_{ij} = \left\{\begin{matrix}
            -\epsilon_{ij} & \quad \text{if }i=k\text{ or }j=k \\
            \epsilon_{ij} + \frac{\epsilon_{ik}|\epsilon_{kj}|+|\epsilon_{ik}|\epsilon_{kj}}{2} & \quad \text{otherwise}
        \end{matrix}\right. .
    \end{equation}
If two seed $\Sigma,\Sigma'$ are connected by a sequence of such isomorphisms $\mu_k$, we say that $\Sigma,\Sigma'$ are \textit{mutation equivalent}.
    
The lattice $\Lambda$ gives rise to a split algebraic torus $\mathcal{X}_\Lambda \coloneqq \text{Hom}(\Lambda,\mathbb{G}_m)$ denoted the \textit{seed} $\mathcal{X}$-torus with elements $X_v\in\mathcal{X}_\Lambda$ for any $v\in\Lambda$. The form $(*,*)$ induces a Poisson structure on $\mathcal{X}_\Lambda$ given by
    \begin{equation}
        \{X_v,X_w\} = (v,w)X_vX_w.
    \end{equation}
The basis $\{e_i\}$ induces a basis $\{X_i=X_{e_i}\}$ in the group of characters of $X_\Lambda$, denoted the \textit{cluster} $\mathcal{X}$ \textit{coordinates}.

Furthermore, the basis $\{e_i\}$ induces a dual basis $\{e_i^*\}$ for the dual lattice $\Lambda^*=\text{Hom}(\Lambda,\mathbb{Z})$. Let $\Lambda^0$ be the sublattice spanned by $f_i=d_i^{-1}e_i^*$. Then we have another split algebraic torus $\mathcal{A}_\Lambda=\text{Hom}(\Lambda^0,\mathbb{G}_m)$ with $\{f_i\}$ providing the basis $\{A_i\}$, denoted the \textit{cluster} $\mathcal{A}$ \textit{coordinates}. There is a natural regular map $p_\Sigma:\mathcal{A}_\Lambda\rightarrow\mathcal{X}_\Lambda$ called the \textit{cluster ensemble map} that translates between cluster $\mathcal{A}$-variables and cluster $\mathcal{X}$-variables. It is given by the formula
    \begin{equation}
        p_\Sigma^*(X_i) = \prod_{j\in I} A_j^{\epsilon_{ji}}.
    \end{equation}

\begin{lemma}
    \cite{FG09} The subtorus $p(\mathcal{A}_\Lambda)$ is a symplectic leaf of the Poisson structure on $\mathcal{X}_\Lambda$.
\end{lemma}

In accordance with standard notation, we will write $\mathcal{X}_\Sigma,\mathcal{A}_\Sigma$ even though these tori only depend on the underlying lattice.

To any seed mutation $\mu_k:\Sigma\rightarrow\Sigma'$, we can associate a pair of birational isomorphisms $\mu_k^\mathcal{A}:\mathcal{A}_\Sigma\rightarrow\mathcal{A}_{\Sigma'}$ and $\mu_k^\mathcal{X}:\mathcal{X}_\Sigma\rightarrow\mathcal{X}_{\Sigma'}$ given by the formulas
    \begin{equation}
    \begin{split}
        (\mu_k^\mathcal{A})^*(A_i') & = \left\{ \begin{matrix}
            A_i & \quad \text{if }i\neq k \\
            A_k^{-1}\left( \prod_{j=1}^mA_j^{[\epsilon_{jk}]_+} + \prod_{j=1}^mA_j^{[-\epsilon_{jk}]_+} \right) & \quad \text{if }i=k
        \end{matrix} \right. \\
        (\mu_k^\mathcal{X})^*(X_i') & = \left\{ \begin{matrix}
            X_iX_k^{[\epsilon_{ki}]_+}(1+X_k)^{-\epsilon_{ki}} & \quad \text{if }i\neq k \\
            X_k^{-1} & \quad \text{if }i=k
        \end{matrix} \right. .
    \end{split}
    \end{equation}
where $[a]_+=\max(a,0)$.

\begin{lemma}
    The canonical ensemble map commutes with cluster variable mutations, i.e.,
        \begin{equation}
            \mu_k^\mathcal{X}\circ p_\Sigma = p_{\mu_k(\Sigma)}\circ \mu_k^\mathcal{A}.
        \end{equation}
\end{lemma}

\subsubsection{Quivers/Amalgamation}

For simplicity, we will denote $\omega_{ij}=(e_i,e_j)$. The combinatorial data of a seed and subsequential mutations can be encoded by a \textit{quiver} $Q$, a planar graph such that $V(Q)=I\sqcup I_0$ with a vertex $i\in I$ for each basis vector $e_i$, a vertex $j\in I_0$ for each frozen index, and arrows $i\rightarrow j$ weighted by the matrix entries $\omega_{ij}$. In this framework, a cluster mutation at vertex $k$ corresponds to a mutation of $Q$ characterized by the following steps:
    \begin{enumerate}
        \item Reverse any arrows incident to $k$,
        \item For any pair of arrows $i\rightarrow k$ and $k\rightarrow j$ with weights $\omega_{ik}$ and $\omega_{kj}$, respectively, draw an arrow $i\rightarrow j$ with weight $\omega_{ij} + \frac{\omega_{ik}\omega_{kj}}{d_k}$,
        \item Delete all arrows with weight $\omega_{ij}=0$ and if there are two arrows $i\rightarrow k$ with weights $\omega_1,\omega_2$, then draw one arrow with weight $\omega_1+\omega_2$.
    \end{enumerate}

\begin{definition}
    Let $Q,Q'$ be two quivers with vertices $V(Q)=I\sqcup I_0, V(Q')=J\sqcup J_0$ and exchange matrices $\epsilon_{ij},\eta_{ij}$ respectively. Let $L$ be a set embedded into both $I_0$ and $J_0$. Then the amalgamation along L is a new quiver $Q''$ with $V(Q'')=K\sqcup K_0$ such that $K=I\cup_LJ, K_0=I_0\cup_LJ_0$ and exchange matrices $\zeta_{ij}$ given by
        \begin{equation}
            \zeta_{ij} = \left\{\begin{matrix}
                0 & \quad \text{if }i\in I-L \text{ and } j\in J-L \\
                0 & \quad \text{if }i\in J-L \text{ and } j\in I-L \\
                \epsilon_{ij} & \quad \text{if }i\in I-L \text{ or } j\in I-L \\
                \eta_{ij} & \quad \text{if }i\in J-L \text{ or } j\in J-L \\
                \epsilon_{ij}+\eta_{ij} & \quad \text{if }i,j\in L
            \end{matrix}\right.
        \end{equation}
\end{definition}

\begin{lemma}
    \cite{FG06} Let $\Sigma,\Sigma',\Sigma''$ be the cluster seeds associated to the quivers $Q,Q',Q''$, respectively. Then amalgamation induces a homomorphism $\mathcal{X}_{\Sigma}\times\mathcal{X}_{\Sigma'}\rightarrow\mathcal{X}_{\Sigma''}$ given by the rule
        \begin{equation}
            Z_i = \left\{\begin{matrix}
                X_i & \quad \text{if } i\in I-L \\
                Y_i & \quad \text{if } i\in J-L \\
                X_iY_i & \quad \text{if } i\in L
            \end{matrix}\right.
        \end{equation}
    Moreover, amalgamation is compatible with both the Poisson and cluster structures.
\end{lemma}

\subsubsection{Quantum Cluster Algebras}

Consider the Heisenberg group $\mathcal{H}_\Lambda$, which is the central extension
    \begin{equation}
        0\rightarrow\mathbb{Z}\rightarrow\mathcal{H}_\Lambda\rightarrow\Lambda\rightarrow0.
    \end{equation}

\begin{definition}
    The \textit{quantum torus algebra} $\mathcal{X}_\Sigma^q$ is the group ring of $\mathcal{H}_\Lambda$. It is identified with the algebra of non-commutative polynomials in $\{X_i\}$ over $\mathbb{Z}[q,q^{-1}]$ with relations
        \begin{equation}
            q^{-\omega_{ij}}X_iX_j = q^{\omega_{ij}}X_jX_i.
        \end{equation}
\end{definition}

We will also denote
    \begin{equation}
        X_{i_1^{m_1},\ldots,i_n^{m_n}} = q^CX_{i_1}^{m_1}\cdots X_{i_n}^{m_n}
    \end{equation}
where $C$ is the unique rational number such that
    \begin{equation}
        q^CX_{i_1}^{m_1}\cdots X_{i_n}^{m_n} = q^{-C}X_{i_n}^{m_n}\cdots X_{i_1}^{m_1}.
    \end{equation}

Given a quantum torus algebra $\mathcal{X}_\Sigma^q$ associated to a seed $\Sigma$, the cluster mutation on index $k$ induces an isomorphism of the skew field of fractions $\mu_k^q:\text{Frac}(\mathcal{X}_{\mu_k(\Sigma)}^q)\rightarrow \text{Frac}(\mathcal{X}_{\Sigma}^q)$ called the \textit{quantum cluster mutation} given by
    \begin{equation}
        \mu_k^q(X_i') = \left\{ \begin{matrix}
            X_k^{-1} & \quad \text{if }i=k, \\
            X_i\prod_{r=1}^{|\epsilon_{ki}|}(1+q_i^{2r-1}X_k) & \quad \text{if }i\neq k\text{ and }\epsilon_{ki}\leq0, \\
            X_i\prod_{r=1}^{\epsilon_{ki}}(1+q_i^{2r-1}X_k^{-1})^{-1} & \quad \text{if }i\neq k\text{ and }\epsilon_{ki}\geq0.
        \end{matrix} \right.
    \end{equation}
The \textit{quantum cluster algebra} $\mathcal{X}^q$ associated to a seed is defined as the subalgebra of $\mathcal{X}_\Sigma^q$ of universally Laurent elements (i.e., remain Laurent polynomials under any combination of finite sequences of cluster mutations).

\subsection{Cluster Structure on Directed Networks}

In \cite{GSV11}, a cluster structure was attached to the directed networks $N_{u,v}$ for Type A. This was generalized in \cite{L21} to give a cluster structure associated to directed networks of any classical type. In this section, we will recall this construction explicitly.

Let $\mathbf{i}$ be an unmixed double reduced Coxeter word for a pair of elements $(u,v)$ in the Weyl group for $G$. Denote $I=\{-n,\ldots,n\}\cup\{1,\ldots,m\}$ the indexing set for the seed $\Sigma_\mathbf{i}$. Furthermore, denote $D_\mathbf{i}$ to be the subnetwork consisting of the bottom $n+1$ rows. For type A, this is the entire network. We can then label the $n+m$ faces of $D_\mathbf{i}$ in the following way. Fix $k\in[1,n]$ and let $\{j:|i_j|=k\}=\{j_1<\cdots<j_r\}$. Then label the $r$ faces between levels $n-1$ and $n+1$ from left to right with $j_1,\ldots,j_r$. 

Now, form the quiver $\Gamma_\mathbf{i}$ with the faces of $D_\mathbf{i}$ as the vertices. Finally, draw arrows between the vertices if the corresponding faces are connected by an edge and exactly one of the two vertices of the edge is either orange or black. Explicitly, draw arrows across edges of the directed network according to \textbf{Figure 5} where whole arrows have weight $\omega_{ij}=1$ and dashed arrows have weight $\omega_{ij}=1/2$.

\begin{figure}[H]
    \centering
    \begin{tikzpicture}
        \tikzset{VertexStyle/.style = {shape = circle,fill = black,minimum size = 1.5mm,inner sep=0pt}}
        \Vertex[color=orange]{A1} \Vertex[x=2,color=black]{B1} \Vertex[x=1,y=-1,color=black]{C1} \Vertex[x=1,y=1,color=black]{D1}
        \Edge[Direct](A1)(B1) \Edge[Direct](D1)(C1)
        \Vertex[x=3,color=black]{A2} \Vertex[x=5,color=orange]{B2} \Vertex[x=4,y=-1,color=black]{C2} \Vertex[x=4,y=1,color=black]{D2}
        \Edge[Direct](A2)(B2) \Edge[Direct](C2)(D2)
        \Vertex[x=6,color=blue]{A3} \Vertex[x=8,color=black]{B3} \Vertex[x=7,y=-1,color=black]{C3} \Vertex[x=7,y=1,color=black]{D3}
        \Edge[Direct](A3)(B3) \Edge[Direct,style={dashed}](D3)(C3)
        \Vertex[x=9,color=blue]{A4} \Vertex[x=11,color=orange]{B4} \Vertex[x=10,y=-1,color=black]{C4} \Vertex[x=10,y=1,color=black]{D4}
        \Edge[Direct](A4)(B4) \Edge[Direct,style={dashed}](C4)(D4)
        \Vertex[x=12,color=black]{A5} \Vertex[x=14,color=red]{B5} \Vertex[x=13,y=-1,color=black]{C5} \Vertex[x=13,y=1,color=black]{D5}
        \Edge[Direct](A5)(B5) \Edge[Direct,style={dashed}](C5)(D5)
        \Vertex[x=15,color=orange]{A6} \Vertex[x=17,color=red]{B6} \Vertex[x=16,y=-1,color=black]{C6} \Vertex[x=16,y=1,color=black]{D6}
        \Edge[Direct](A6)(B6) \Edge[Direct,style={dashed}](D6)(C6)
    \end{tikzpicture}
    \caption{Rules to obtain Cluster Quiver from Directed Network}
    \label{fig:5}
\end{figure}
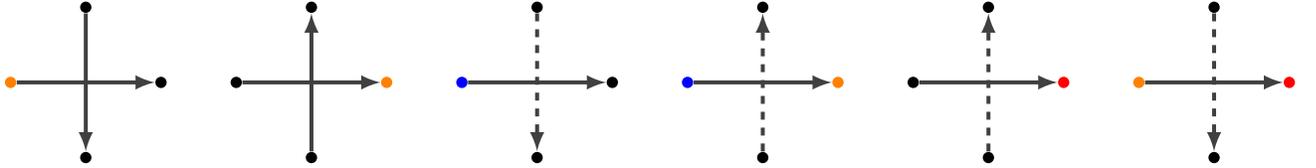

The quotient $G^{u,v}/H$ corresponds to drawing $D_\mathbf{i}$ on a cylinder to obtain $\widetilde{D}_\mathbf{i}$ \cite[Prop. 3.4]{L21}. This is due to the following observation: let $hgh^{-1}\in G^{u,v}$ for $h\in H$ be written as a directed network $D$. Then by writing $D$ on a cylinder, we are identifying the left and right ends, which allows us to write $hgh^{-1}$ as $h^{-1}hg=g$ for any $h\in H$. Thus, passing $D_\mathbf{i}$ to $\widetilde{D}_\mathbf{i}$ corresponds to the projection $G^{u,v}\rightarrow G^{u,v}/H$. 

In doing this, we obtain a new quiver $Q_\mathbf{i}$ by amalgamating the left-most and right-most vertices lying in the same row. This gives us the exchange matrix $\epsilon$ having entries in $\{-n,\ldots,n\}$ where $\epsilon_{ij}>0$ if $i\rightarrow j$, $\epsilon_{ij}<0$ if $j\rightarrow i$, and $\epsilon_{ij}=0$ if there are no edges connecting $i$ and $j$. Moreover,
    \begin{equation}
        |\epsilon_{jk}| = \left\{ \begin{matrix}
            2 & \quad \text{if }|j|=|k|, \\
            -C_{|i_j|,|i_k|}\times\#\{\text{arrows connecting }v_j\text{ and }v_k\} & \quad \text{if }|j|\neq|k|
        \end{matrix} \right.
    \end{equation}
where $C$ is the Cartan matrix. Explicitly, for the standard double Coxeter word $\mathbf{i_0}=(-1,\ldots,-n,1,\ldots,n)$, it follows
    \begin{equation}
        \epsilon_{\mathbf{i_0}} = \begin{pmatrix}
            0 & C \\
            -C & 0
        \end{pmatrix}
    \end{equation}
with the columns labelled $-1,\ldots,n,1,\ldots n$.

Given the exchange matrix $\epsilon_{ij}$, the edge weights for $Q_{\mathbf{i}}$ are $\omega_{ij}=\epsilon_{ij}d_j$ (our convention has $d_j$ instead of $d_j^{-1}$ to force all edge weights to have integer values). Explicitly, the quivers $Q_{\mathbf{i_0}}$ for types A and C are given in \textbf{Figure 6}.

\begin{figure}
    \centering
    \begin{subfigure}[b]{0.3\textwidth}
        \centering
        \begin{tikzpicture}
            \tikzset{VertexStyle/.style = {shape = circle,fill = black,minimum size = 3mm,inner sep=0pt}}
            \Vertex[color=black]{A} \Vertex[x=2,color=black]{B} \Vertex[y=-0.05,Pseudo]{A1} \Vertex[x=2,y=-0.05,Pseudo]{B1} \Vertex[y=0.05,Pseudo]{A2} \Vertex[x=2,y=0.05,Pseudo]{B2} \Edge[Direct](A1)(B1) \Edge[Direct](A2)(B2)
            \Vertex[y=2,color=black]{C} \Vertex[x=2,y=2,color=black]{D} \Vertex[y=1.95,Pseudo]{C1} \Vertex[x=2,y=1.95,Pseudo]{D1} \Vertex[y=2.05,Pseudo]{C2} \Vertex[x=2,y=2.05,Pseudo]{D2} \Edge[Direct](C1)(D1) \Edge[Direct](C2)(D2) \Edge[Direct](B)(C) \Edge[Direct](D)(A)
            \Vertex[x=1,y=2.5,color=black,size=0.2]{X1}
            \Vertex[x=1,y=3,color=black,size=0.2]{X2}
            \Vertex[x=1,y=3.5,color=black,size=0.2]{X3}
            \Vertex[y=4,color=black]{E} \Vertex[x=2,y=4,color=black]{F} \Vertex[y=3.95,Pseudo]{E1} \Vertex[x=2,y=3.95,Pseudo]{F1} \Vertex[y=4.05,Pseudo]{E2} \Vertex[x=2,y=4.05,Pseudo]{F2} \Edge[Direct](E1)(F1) \Edge[Direct](E2)(F2)
            \Vertex[y=6,color=black]{G} \Vertex[x=2,y=6,color=black]{H} \Vertex[y=5.95,Pseudo]{G1} \Vertex[x=2,y=5.95,Pseudo]{H1} \Vertex[y=6.05,Pseudo]{G2} \Vertex[x=2,y=6.05,Pseudo]{H2} \Edge[Direct](G1)(H1) \Edge[Direct](G2)(H2) \Edge[Direct](F)(G) \Edge[Direct](H)(E)
            \Text[x=-0.75]{-1} \Text[x=2.75]{1} \Text[x=-0.75,y=2]{-2} \Text[x=2.75,y=2]{2} \Text[x=-0.75,y=4]{-(n-1)} \Text[x=2.75,y=4]{n-1} \Text[x=-0.75,y=6]{-n} \Text[x=2.75,y=6]{n}
        \end{tikzpicture}
        \caption{Type A}
        \label{fig:6a}
    \end{subfigure}
    \hspace{15mm}
    \begin{subfigure}[b]{0.3\textwidth}
        \centering
        \begin{tikzpicture}
            \tikzset{VertexStyle/.style = {shape = circle,fill = black,minimum size = 3mm,inner sep=0pt}}
            \Vertex[color=black]{A} \Vertex[x=2,color=black]{B} \Vertex[y=-0.05,Pseudo]{A1} \Vertex[x=2,y=-0.05,Pseudo]{B1} \Vertex[y=0.05,Pseudo]{A2} \Vertex[x=2,y=0.05,Pseudo]{B2} \Edge[Direct](A1)(B1) \Edge[Direct](A2)(B2)
            \Vertex[y=2,color=black]{C} \Vertex[x=2,y=2,color=black]{D} \Vertex[y=1.95,Pseudo]{C1} \Vertex[x=2,y=1.95,Pseudo]{D1} \Vertex[y=2.05,Pseudo]{C2} \Vertex[x=2,y=2.05,Pseudo]{D2} \Edge[Direct](C1)(D1) \Edge[Direct](C2)(D2) \Edge[Direct](B)(C) \Edge[Direct](D)(A)
            \Vertex[x=1,y=2.5,color=black,size=0.2]{X1}
            \Vertex[x=1,y=3,color=black,size=0.2]{X2}
            \Vertex[x=1,y=3.5,color=black,size=0.2]{X3}
            \Vertex[y=4,color=black]{E} \Vertex[x=2,y=4,color=black]{F} \Vertex[y=3.95,Pseudo]{E1} \Vertex[x=2,y=3.95,Pseudo]{F1} \Vertex[y=4.05,Pseudo]{E2} \Vertex[x=2,y=4.05,Pseudo]{F2} \Edge[Direct](E1)(F1) \Edge[Direct](E2)(F2)
            \Vertex[y=6,color=black]{G} \Vertex[x=2,y=6,color=black]{H} \Vertex[y=5.95,Pseudo]{G1} \Vertex[x=2,y=5.95,Pseudo]{H1} \Vertex[y=6.05,Pseudo]{G2} \Vertex[x=2,y=6.05,Pseudo]{H2} \Edge[Direct](G1)(H1) \Edge[Direct](G2)(H2) \Edge[Direct](F)(G) \Edge[Direct](H)(E)
            \Vertex[y=8,color=black]{I} \Vertex[x=2,y=8,color=black]{J}
            \Edge[Direct,label=\textbf{4}](I)(J)
            \Vertex[x=2.1,y=6,Pseudo]{B3} \Vertex[x=2,y=5.95,Pseudo]{B4}
            \Vertex[x=0.05,y=8.05,Pseudo]{C3} \Vertex[x=-0.05,y=8,Pseudo]{C4}
            \Edge[Direct](B3)(C3) \Edge[Direct](B4)(C4)
            \Vertex[x=-0.1,y=6,Pseudo]{A3} \Vertex[y=5.95,Pseudo]{A4}
            \Vertex[x=2,y=8.1,Pseudo]{D3} \Vertex[x=2.05,y=8,Pseudo]{D4}
            \Edge[Direct](D3)(A3) \Edge[Direct](D4)(A4)
            \Text[x=-0.75]{-1} \Text[x=2.75]{1} \Text[x=-0.75,y=2]{-2} \Text[x=2.75,y=2]{2} \Text[x=-0.75,y=4]{-(n-2)} \Text[x=2.75,y=4]{n-2} \Text[x=-0.75,y=6]{-(n-1)} \Text[x=2.75,y=6]{n-1} \Text[x=-0.75,y=8]{-n} \Text[x=2.75,y=8]{n} 
        \end{tikzpicture}
        \caption{Type C}
        \label{fig:6b}
    \end{subfigure}
    \caption{Cluster Quivers for $\mathbf{i_0}$}
    \label{fig:6}
\end{figure}
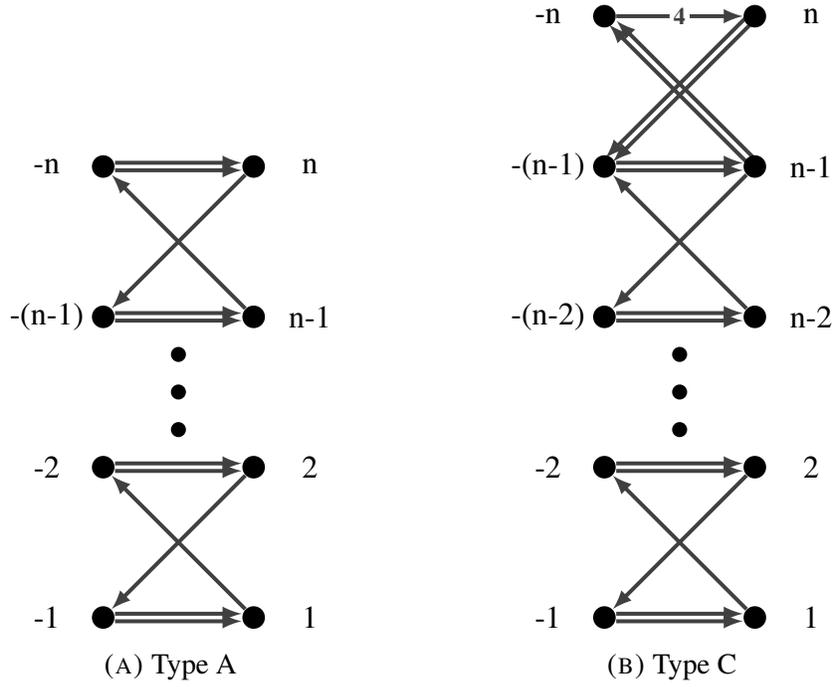

\begin{remark}
    Using the convention of \cite{L21}, a mutation at a vertex $-k$ will be paired with a permutation of indices in the following way:
        \begin{equation}
            \tau_k = \mu_{-k}\circ\sigma_k
        \end{equation}
    where $\mu_{-k}$ is a cluster mutation at vertex $-k$ and $\sigma_k$ is the permutation such that
        \begin{equation}
            \sigma_k(j) = \left\{\begin{matrix}
                j & \quad \text{if }|j|\neq k \\
                -j & \quad \text{if }|j|=k
            \end{matrix} \right.
        \end{equation}
\end{remark}

Consider the 3 \textit{quiver blocks} shown in \textbf{Figure 7}. For a quiver of Type $A_n$ or $C_n$, there will be $n-1$ quiver blocks $Q_1,\ldots,Q_{n-1}$ labelled from bottom to top glued together to make the full quiver $Q_\mathbf{i}$. The following definition will be useful:

\begin{definition}
    Let $\mathbf{i}$ be a double unmixed Coxeter word for a type $A_n$ Coxeter-Toda system. We define the \textit{quiver vector associated to $\mathbf{i}$}, $\vec{Q}_{n-1}=(Q_{n-1},\ldots,Q_1)\in\{-1,0,1\}^{n-1}$, in the following way: if the quiver block $Q_i$ is of type 7(a), 7(b), or 7(c), then the corresponding entry $Q_i=0,1$, or $-1$, respectively.
\end{definition}

\begin{figure}
    \centering
    \begin{subfigure}[b]{0.3\textwidth}
        \centering
        \begin{tikzpicture}
            \tikzset{VertexStyle/.style = {shape = circle,fill = black,minimum size = 3mm,inner sep=0pt}}
            \Vertex[color=black]{A} \Vertex[x=2,color=black]{B} \Vertex[y=-0.05,Pseudo]{A1} \Vertex[x=2,y=-0.05,Pseudo]{B1} \Vertex[y=0.05,Pseudo]{A2} \Vertex[x=2,y=0.05,Pseudo]{B2} \Edge[Direct](A1)(B1) \Edge[Direct](A2)(B2)
            \Vertex[y=2,color=black]{C} \Vertex[x=2,y=2,color=black]{D} \Vertex[y=1.95,Pseudo]{C1} \Vertex[x=2,y=1.95,Pseudo]{D1} \Vertex[y=2.05,Pseudo]{C2} \Vertex[x=2,y=2.05,Pseudo]{D2} \Edge[Direct](C1)(D1) \Edge[Direct](C2)(D2) \Edge[Direct](B)(C) \Edge[Direct](D)(A)
        \end{tikzpicture}
        \caption{$Q_i=0$}
        \label{fig:7a}
    \end{subfigure}
    \hspace{1mm}
    \begin{subfigure}[b]{0.3\textwidth}
        \centering
        \begin{tikzpicture}
            \tikzset{VertexStyle/.style = {shape = circle,fill = black,minimum size = 3mm,inner sep=0pt}}
            \Vertex[color=black]{A} \Vertex[x=2,color=black]{B} \Vertex[y=-0.05,Pseudo]{A1} \Vertex[x=2,y=-0.05,Pseudo]{B1} \Vertex[y=0.05,Pseudo]{A2} \Vertex[x=2,y=0.05,Pseudo]{B2} \Edge[Direct](A1)(B1) \Edge[Direct](A2)(B2)
            \Vertex[y=2,color=black]{C} \Vertex[x=2,y=2,color=black]{D} \Vertex[y=1.95,Pseudo]{C1} \Vertex[x=2,y=1.95,Pseudo]{D1} \Vertex[y=2.05,Pseudo]{C2} \Vertex[x=2,y=2.05,Pseudo]{D2} \Edge[Direct](C1)(D1) \Edge[Direct](C2)(D2) \Edge[Direct](C)(A) \Edge[Direct](D)(B)
            \Vertex[x=2.1,Pseudo]{B3} \Vertex[x=2,y=-0.05,Pseudo]{B4}
            \Vertex[x=0.05,y=2.05,Pseudo]{C3} \Vertex[x=-0.05,y=2,Pseudo]{C4}
            \Edge[Direct](B3)(C3) \Edge[Direct](B4)(C4)
        \end{tikzpicture}
        \caption{$Q_i=1$}
        \label{fig:7b}
    \end{subfigure}
    \hspace{1mm}
    \begin{subfigure}[b]{0.3\textwidth}
        \centering
        \begin{tikzpicture}
            \tikzset{VertexStyle/.style = {shape = circle,fill = black,minimum size = 3mm,inner sep=0pt}}
            \Vertex[color=black]{A} \Vertex[x=2,color=black]{B} \Vertex[y=-0.05,Pseudo]{A1} \Vertex[x=2,y=-0.05,Pseudo]{B1} \Vertex[y=0.05,Pseudo]{A2} \Vertex[x=2,y=0.05,Pseudo]{B2} \Edge[Direct](A1)(B1) \Edge[Direct](A2)(B2)
            \Vertex[y=2,color=black]{C} \Vertex[x=2,y=2,color=black]{D} \Vertex[y=1.95,Pseudo]{C1} \Vertex[x=2,y=1.95,Pseudo]{D1} \Vertex[y=2.05,Pseudo]{C2} \Vertex[x=2,y=2.05,Pseudo]{D2} \Edge[Direct](C1)(D1) \Edge[Direct](C2)(D2) \Edge[Direct](A)(C) \Edge[Direct](B)(D)
            \Vertex[x=-0.1,Pseudo]{A3} \Vertex[y=-0.05,Pseudo]{A4}
            \Vertex[x=2,y=2.1,Pseudo]{D3} \Vertex[x=2.05,y=2,Pseudo]{D4}
            \Edge[Direct](D3)(A3) \Edge[Direct](D4)(A4)
        \end{tikzpicture}
        \caption{$Q_i=-1$}
        \label{fig:7c}
    \end{subfigure}
    \caption{Possible Quiver Blocks}
    \label{fig:7}
\end{figure}
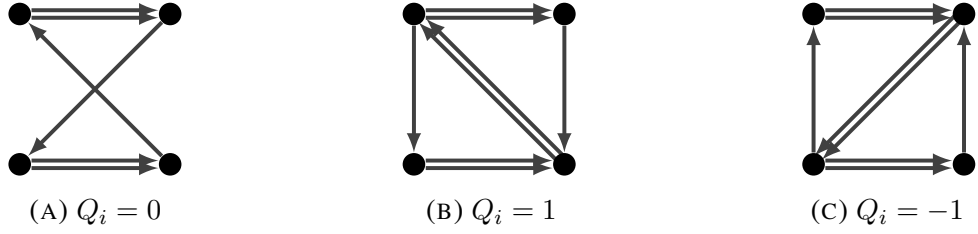

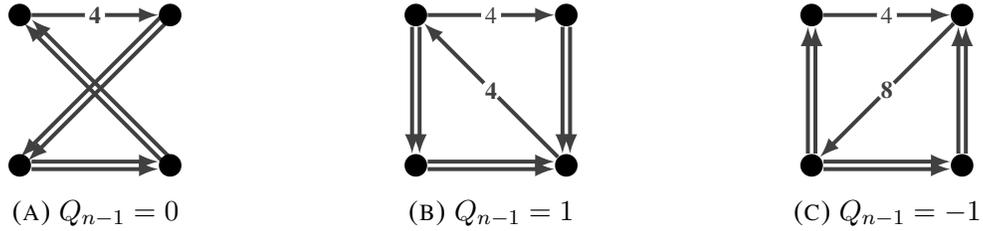
\begin{figure}
    \centering
    \begin{subfigure}[b]{0.3\textwidth}
        \centering
        \begin{tikzpicture}
            \tikzset{VertexStyle/.style = {shape = circle,fill = black,minimum size = 3mm,inner sep=0pt}}
            \Vertex[color=black]{A} \Vertex[x=2,color=black]{B} \Vertex[y=-0.05,Pseudo]{A1} \Vertex[x=2,y=-0.05,Pseudo]{B1} \Vertex[y=0.05,Pseudo]{A2} \Vertex[x=2,y=0.05,Pseudo]{B2} \Edge[Direct](A1)(B1) \Edge[Direct](A2)(B2)
            \Vertex[y=2,color=black]{C} \Vertex[x=2,y=2,color=black]{D} \Edge[Direct,label=\textbf{4}](C)(D)
            \Vertex[x=2.1,Pseudo]{B3} \Vertex[x=2,y=-0.05,Pseudo]{B4}
            \Vertex[x=0.05,y=2.05,Pseudo]{C3} \Vertex[x=-0.05,y=2,Pseudo]{C4}
            \Edge[Direct](B3)(C3) \Edge[Direct](B4)(C4)
            \Vertex[x=-0.1,Pseudo]{A3} \Vertex[y=-0.05,Pseudo]{A4}
            \Vertex[x=2,y=2.1,Pseudo]{D3} \Vertex[x=2.05,y=2,Pseudo]{D4}
            \Edge[Direct](D3)(A3) \Edge[Direct](D4)(A4)
        \end{tikzpicture}
        \caption{$Q_{n-1}=0$}
        \label{fig:8a}
    \end{subfigure}
    \hspace{1mm}
    \begin{subfigure}[b]{0.3\textwidth}
        \centering
        \begin{tikzpicture}
            \tikzset{VertexStyle/.style = {shape = circle,fill = black,minimum size = 3mm,inner sep=0pt}}
            \Vertex[color=black]{A} \Vertex[x=2,color=black]{B} \Vertex[y=-0.05,Pseudo]{A1} \Vertex[x=2,y=-0.05,Pseudo]{B1} \Vertex[y=0.05,Pseudo]{A2} \Vertex[x=2,y=0.05,Pseudo]{B2} \Vertex[x=-0.05,Pseudo]{A3} \Vertex[x=1.95,Pseudo]{B3} \Vertex[x=0.05,Pseudo]{A4} \Vertex[x=2.05,Pseudo]{B4} \Edge[Direct](A1)(B1) \Edge[Direct](A2)(B2)
            \Vertex[y=2,color=black]{C} \Vertex[x=2,y=2,color=black]{D} \Vertex[x=-0.05,y=2,Pseudo]{C1} \Vertex[x=1.95,y=2,Pseudo]{D1} \Vertex[x=0.05,y=2,Pseudo]{C2} \Vertex[x=2.05,y=2,Pseudo]{D2} \Edge[Direct,label=4](C)(D) \Edge[Direct](C1)(A3) \Edge[Direct](C2)(A4) \Edge[Direct](D1)(B3) \Edge[Direct](D2)(B4) \Edge[Direct,label=\textbf{4}](B)(C)
        \end{tikzpicture}
        \caption{$Q_{n-1}=1$}
        \label{fig:8b}
    \end{subfigure}
    \hspace{1mm}
    \begin{subfigure}[b]{0.3\textwidth}
        \centering
        \begin{tikzpicture}
            \tikzset{VertexStyle/.style = {shape = circle,fill = black,minimum size = 3mm,inner sep=0pt}}
            \Vertex[color=black]{A} \Vertex[x=2,color=black]{B} \Vertex[y=-0.05,Pseudo]{A1} \Vertex[x=2,y=-0.05,Pseudo]{B1} \Vertex[y=0.05,Pseudo]{A2} \Vertex[x=2,y=0.05,Pseudo]{B2} \Vertex[x=-0.05,Pseudo]{A3} \Vertex[x=1.95,Pseudo]{B3} \Vertex[x=0.05,Pseudo]{A4} \Vertex[x=2.05,Pseudo]{B4} \Edge[Direct](A1)(B1) \Edge[Direct](A2)(B2)
            \Vertex[y=2,color=black]{C} \Vertex[x=2,y=2,color=black]{D} \Vertex[x=-0.05,y=2,Pseudo]{C1} \Vertex[x=1.95,y=2,Pseudo]{D1} \Vertex[x=0.05,y=2,Pseudo]{C2} \Vertex[x=2.05,y=2,Pseudo]{D2} \Edge[Direct,label=4](C)(D) \Edge[Direct](A3)(C1) \Edge[Direct](A4)(C2) \Edge[Direct](B3)(D1) \Edge[Direct](B4)(D2) \Edge[Direct,label=\textbf{8}](D)(A)
        \end{tikzpicture}
        \caption{$Q_{n-1}=-1$}
        \label{fig:8c}
    \end{subfigure}
    \caption{Possible $Q_{n-1}$ Quiver Blocks for $C_n$}
    \label{fig:8}
\end{figure}

\begin{lemma}
    There are exactly $3^{n-1}$ double reduced Coxeter words $(u,v)$ when $u,v$ both have length $n$. Furthermore, each double reduced Coxeter word corresponds to a quiver, giving $3^{n-1}$ quivers for types $A_n$ and $C_n$.
\end{lemma}

\begin{proof}
    Consider the reflection $s_i \in W$ given to a letter $i$ when $i>0$ and $\overline{s}_{-i}$ if $i<0$. The relations between letters are given by
        \begin{equation}
        \begin{split}
            & s_is_j=s_js_i \quad \text{if} \quad |i-j|>1 \\
            & \overline{s}_i\overline{s}_j=\overline{s}_j\overline{s}_i \quad \text{if} \quad |i-j|>1 \\
            & s_i\overline{s}_j=\overline{s}_js_i \quad \text{if} \quad |i-j|\neq0.
        \end{split}
        \end{equation}
    We will proceed by induction on $n$. Let $W_n$ be the group of double Coxeter Weyl words of length $2n$. Consider the base case $n=1$. It follows
        \begin{equation}
            s_1\overline{s}_1
        \end{equation}
    is the unique element of $W_1$. Now, suppose the proposition holds for $W_k$. Let $w\in W_k$. In general, we can write $w$ in the form
        \begin{equation}
            w=s_k\cdots \overline{s}_k\cdots
        \end{equation}
    where we can always put $s_k$ as the first entry by the 3rd relation stated above. Then there are 3 ways to distinctly place $s_{k+1}$ and $\overline{s}_{k+1}$ given by
        \begin{equation}
        \begin{split}
            & s_k\cdots s_{k+1}\overline{s}_{k+1}\cdots\overline{s}_k\cdots, \\
            & s_k\cdots\overline{s}_{k+1}s_{k+1}\cdots\overline{s}_k\cdots, \\
            & s_k\cdots\overline{s}_k\cdots\overline{s}_{k+1}s_{k+1}\cdots
        \end{split}
        \end{equation}
    Thus, we have 3 ways to distinctly place $s_{k+1}$ and $\overline{s}_{k+1}$ in any $w\in W_k$, and there are $3^{k-1}$ such $w$'s by the inductive hypothesis. Hence, it follows that there are $3\cdot3^{k-1}=3^{k}$ distinct words in $W_{k+1}$, proving the inductive step.

    Furthermore, it is a direct consequence from the quiver construction that the 3 possible double Coxeter words above correspond to $Q_k=0,1,-1$, respectively, proving the second statement in the lemma. Therefore, there are $3^{n-1}$ possible possible Type $A_n$ and $C_n$ quivers.
\end{proof}

\begin{theorem}
    \cite{L21} Given any two double reduced Coxeter words $\mathbf{i}$ and $\mathbf{j}$, the quivers $Q_\mathbf{i}$ and $Q_\mathbf{j}$ are mutation equivalent.
\end{theorem}

Combining Lemma 4.9 and Theorem 4.10, all $3^{n-1}$ possible quivers (i.e., cluster seeds) corresponding to a double reduced Coxeter word are mutation equivalent.

\subsection{Poisson Structure \& Quantization}

In this section we will first recall the Poisson brackets of edge weights on a given direct network in \cite{L21}. We then quantize these Poisson brackets. In Section 5, we prove that the quantized Coxeter-Toda Hamiltonians obtained are precisely the ones obtained by the Lax formalism in \cite{FT19},\cite{GT19} by counting quantum path weights inductively.

\begin{lemma}
    \cite{L21} Let $(u,v)$ be a pair of Coxeter elements and $\mathbf{i}=(i_1,\ldots,i_{2n})$ be a double reduced word for $(u,v)$. Then with respect to the factorization
        \begin{equation}
            \overline{g} = E_{i_1}(1)\cdots E_{i_n}(1)D(t_1,\ldots,t_n)E_{i_{n+1}}(c_{i_{n+1}})\cdots E_{i_{2n}}(c_{i_{2n}})
        \end{equation}
    of $g\in G^{u,v}/H$, the Poisson brackets between the rational functions $c_j,t_k$ for $j,k\in[1,n]$ are given by
        \begin{equation}
        \begin{split}
            \{c_j,c_k\} & = 2\omega_{ij}c_jc_k \\
            \{c_j,t_k\} & = 2d_j\delta_{jk}c_jt_k \\
            \{t_j,t_k\} & = 0.
        \end{split}
        \end{equation}
\end{lemma}

Now, let $\mathcal{A}_\mathbf{i},\mathcal{X}_\mathbf{i}$ be the cluster seeds coming from the quiver $Q_\mathbf{i}$. Recall the induced Poisson structure on $\mathcal{A}_\mathbf{i}$ from $\mathcal{X}_\mathbf{i}$:
    \begin{equation}
        \{A_i,A_j\} = \omega_{ij}A_iA_j.
    \end{equation}

\begin{theorem}
    \cite{L21} There is a Poisson map $a_\mathbf{i}:\mathcal{A}_\mathbf{i}\rightarrow G^{u,v}/H$ given by
        \begin{equation}
        \begin{split}
            a_\mathbf{i}^*(t_j) & = A_{-j}A_j^{-1} \\
            a_\mathbf{i}^*(c_j) & = \prod_{k\in I} A_k^{-\epsilon_{kj}}.
        \end{split}
        \end{equation}
\end{theorem}

\begin{definition}
    Using $a_\mathbf{i}$ and quantizing the Poisson brackets above, the quantum torus algebra, $\mathcal{X}_\mathbf{i}^q$, can be simultaneously defined as the associative algebra over $\mathbb{C}(q^d)$, with $d=\min_{j\in I}(d_j)$, defined by generators $\{X_{t_j}^{\pm1},X_{c_j}^{\pm1}\}_{j\in\widetilde{I}}$ and relations
        \begin{equation}
        \begin{split}
            X_{c_j}X_{c_k} & = q^{-2\omega_{ij}}X_{c_k}X_{c_j} \\
            X_{c_j}X_{t_k} & = q^{-2\delta_{jk}d_j}X_{t_k}X_{c_j} \\
            X_{t_j}X_{t_k} & = X_{t_k}X_{t_j}.
        \end{split}
        \end{equation}
    Furthermore, we will denote
        \begin{equation}
            X_{a_1\cdots a_m} = q^CX_{a_1}\cdots X_{a_m}, \quad a_k\in\{t_j,c_j\}_{j\in I}
        \end{equation}
    where $C$ is the unique rational number such that
        \begin{equation}
            q^CX_{a_1}\cdots X_{a_m} = q^{-C}X_{a_m}\cdots X_{a_1}.
        \end{equation}
\end{definition}

The convention we will use for quantization of paths on the network $N_{u,v}(\mathbf{i})$ will be the following: if the weights collected along a single path are given by $a_1,\ldots,a_m$, then the quantized weight of the whole path will be $X_{a_1\cdots a_m}$. Moreover, the quantized weight corresponding to a family of non-intersecting paths with quantized weights $X_1,\ldots,X_m$ is $X_1\cdots X_m$ with the convention that multiplication is done from top to bottom along $N_{u,v}(\mathbf{i})$.

\begin{remark}
    It will be convenient for us to label quantized path weights as $X_{j,k}$ where $k$ is the row of the source (and sink) of the path and $j$ is the lowest row intersecting the path.
\end{remark}

\subsection{Coxeter-Toda Hamiltonians}

Given a directed network for an unmixed double Coxeter word $\mathbf{i}$, we can reproduce the matrix $\bar{g}$. The $(i,j)$-th entry of $\bar{g}$ is the sum of path weights over all paths from source $i$ to sink $j$. We will be using these networks to recover the Hamiltonians for $\bar{g}$:

\begin{theorem}
    \cite{L21} The Coxeter-Toda Hamiltonians for a simple complex Lie group $G$ of dimension $n$ and a choice of Coxeter words $u,v\in W$ is given by
        \begin{equation}
            H_j^{u,v} = \sum_{I\subseteq[1,n],|I|=j} \sum_{P\in P^{u,v}_{ni}(I)} wt(P)
        \end{equation}
    where $P^{u,v}_{ni}(I)$ is the set consisting of families of non-intersecting paths such that the set of sources and the set of sinks of all paths in the family is equal to $I$.
\end{theorem}

\section{$\mathbf{2\times2}$ Lax Formulation}

\subsection{Type A}

The following Lax matrix formulation for the type A q-Toda system is given in \cite{FT19}. Let $v$ be an indeterminate and consider the associative $\mathbb{C}(v)$-algebra $\mathcal{A}_n^v$ generated by $\{w_i^{\pm1},D_i^{\pm1}\}_{i=1}^n$ with defining relations
    \begin{equation}
        [w_i,w_j] = [D_i,D_j] = 0, \quad w_i^{\pm1}w_i^{\mp1} = D_i^{\pm1}D_i^{\mp1} = 1, \quad D_iw_j = v^{\delta_{ij}}w_jD_i
    \end{equation}
Define 3 (local) trigonometric Lax matrices:
    \begin{equation}
    \begin{split}
        L_i^{v,0}(z) & = \begin{pmatrix}
            w_i^{-1}z^{1/2} - w_iz^{-1/2} & D_i^{-1}z^{1/2} \\
            -D_iz^{-1/2} & 0
        \end{pmatrix} \\
        L_i^{v,-1}(z) & = \begin{pmatrix}
            w_i^{-1} - w_iz^{-1} & w_iD_i^{-1} \\
            -w_iD_iz^{-1} & w_i
        \end{pmatrix} \\
        L_i^{v,1}(z) & = \begin{pmatrix}
            w_i^{-1}z - w_i & w_i^{-1}D_i^{-1}z \\
            -w_i^{-1}D_i & -w_i^{-1}
        \end{pmatrix}.
    \end{split}
    \end{equation}

\begin{lemma} \cite{FT19}
    The 3 Lax matrices above satisfy the trigonometric RTT relation with the standard trigonometric $R$-matrix
        \begin{equation}
            R_{\text{trig}}(z) = \begin{pmatrix}
                1 & 0 & 0 & 0 \\
                0 & \frac{z-1}{vz-v^{-1}} & \frac{z(v-v^{-1})}{vz-v^{-1}} & 0 \\
                0 & \frac{v-v^{-1}}{vz-v^{-1}} & \frac{z-1}{vz-v^{-1}} & 0 \\
                0 & 0 & 0 & 1 \\
            \end{pmatrix}
        \end{equation}
\end{lemma}

Now, let $\vec{k}_n = (k_n,\ldots,k_1)\in\{-1,0,1\}^n$ be an index vector. Define the mixed complete monodromy matrix
    \begin{equation}
        T_{\vec{k}_n}^v(z) = L_n^{v,k_n}(z)\cdots L_1^{v,k_1}(z).
    \end{equation}
    
\begin{remark}
    It follows from the above lemma that the mixed complete monodromy matrix $T_{\vec{k}_n}^v(z)$ satisfies the trigonometric RTT relation with R-matrix given by $R_{\text{trig}}(z)$. Therefore, the coefficients in $z$ of $T_{\vec{k}_n}^v(z)_{11}$ generate a commutative subalgebra of $\mathcal{A}_n^v$.
\end{remark}

Explicitly, we have
    \begin{equation}
        T_{\vec{k}_n}^v(z)_{11} = H_1^{\vec{k}_n}z^{\sigma_n} + H_2^{\vec{k}_n}z^{\sigma_n+1} + \cdots + H_{n+1}^{\vec{k}_n}z^{\sigma_n+n}
    \end{equation}
where
    \begin{equation}
        \sigma_n = \sum_{i=1}^n s_i, \quad s_i = \frac{k_i-1}{2}.
    \end{equation}
    
\begin{proposition}
    \cite{FT19} Let $\vec{k}_n'=(0,k_{n-1},\ldots,k_2,0)$. Then $H_2^{\vec{k}_n} = H_2^{\vec{k}_n'}$.
\end{proposition}

The above proposition shows that there are at most $3^{n-2}$ different q-Toda systems given by the above Lax formalism. Furthermore, these Hamiltonians are identified with the type $A_{n-1}$ q-Toda Hamiltonians given in \cite{E99},\cite{S99}. 

\subsection{Type C}

The following construction is due to \cite{GT19}. To start, define 3 more (local) trigonometric Lax matrices:
    \begin{equation}
    \begin{split}
        \bar{L}_i^{v,0}(z) & = \begin{pmatrix}
            w_iz^{1/2} - w_i^{-1}z^{-1/2} & D_iz^{1/2} \\
            -D_i^{-1}z^{-1/2} & 0
        \end{pmatrix} \\
        \bar{L}_i^{v,-1}(z) & = \begin{pmatrix}
            w_i - w_i^{-1}z^{-1} & w_i^{-1}D_i \\
            -w_i^{-1}D_i^{-1}z^{-1} & w_i^{-1}
        \end{pmatrix} \\
        \bar{L}_i^{v,1}(z) & = \begin{pmatrix}
            w_iz - w_i^{-1} & w_iD_iz \\
            -w_iD_i^{-1} & -w_i
        \end{pmatrix}.
    \end{split}
    \end{equation}
Furthermore, given an index vector $\vec{k}_n=(k_n,\ldots,k_1)$, define the double complete mixed monodromy matrix:
    \begin{equation}
        \mathbb{T}_{\vec{k}_n}^v(z) = \bar{L}_1^{v,-k_1}(z)\cdots\bar{L}_n^{v,-k_n}(z)L_n^{v,k_n}(z)\cdots L_1^{v,k_1}(z).
    \end{equation}
The following theorem is the result of a direct calculation.

\begin{theorem}
    \cite{GT19} The double complete mixed monodromy matrix $\mathbb{T}_{\vec{k}_n}^v(z)$ satisfies the RTT relation with R-matrix given by $R_{\text{trig}}(z)$. Therefore, the coeffients in $z$ of $\mathbb{T}_{\vec{k}_n}^v(z)_{11}$, given by
        \begin{equation}
            \mathbb{T}_{\vec{k}_n}^v(z)_{11} = \mathbb{H}_1^{\vec{k}_n}z^{-n} + \mathbb{H}_2^{\vec{k}_n}z^{-n+1} + \cdots + \mathbb{H}_{2n+1}^{\vec{k}_n}z^{n}
        \end{equation}
    form a commutative subalgebra of $\mathcal{A}_n^v$. Furthermore, this commutative subalgebra offers a Lax matrix realization of the type $C_n$ modified q-Toda system.
\end{theorem}

\section{Calculation of q-Toda Hamiltonians}

\subsection{Type A}

In this section, we provide a recursive formula to obtain q-Toda Hamiltonians based on the Lax formulation in \cite{FT19}. Then we will use this formula to prove a bijection between the q-Toda Hamiltonians obtained via directed networks and Lax matrices. Thus, we begin by stating the recursive formula:

\begin{theorem}
    Let $\vec{k}_{n+1}\in\{-1,0,1\}^{n+1}$ be an index vector for the mixed complete monodromy matrix $T_{\vec{k}_{n+1}}^v(z)$. Then the $i$-th $A_n$ q-Toda Hamiltonian associated to $\vec{k}_{n+1}$ can be written as
        \begin{equation}
        \begin{split}
            H_i^{\vec{k}_{n+1}} & = -w_{n+1}H_i^{\vec{k}_n} + w_{n+1}^{-1}H_{i-1}^{\vec{k}_n} - \sigma_{n,n+1}D_nD_{n+1}^{-1}H_{i-1}^{\vec{k}_{n-1}} \\
            & \quad + \sum_{m=0}^{n-2}(-1)^{n-j} k_{m+2,n}\sigma_{m+1,n+1}D_{m+1}D_{n+1}^{-1}H^{\vec{k}_m}_{i-1-S_{n,m+1}}
        \end{split}
        \end{equation}
    where
        \begin{equation}
            S_{l,m} = S_l - S_m, \quad S_j = \sum_{i=1}^j s_i, \quad s_i = \frac{k_i-1}{2}.
        \end{equation}
    and
        \begin{equation}
            k_{i,j} = k_ik_{i+1}\cdots k_j, \quad \sigma_{i,j} = w_i^{-k_i}w_{i+1}^{-k_{i+1}}\cdots w_j^{-k_j}.
        \end{equation}
\end{theorem}

\begin{proof}
    As a preliminary result, we claim that the type $A_{n-1}$ mixed complete monodromy matrix associated to an index vector $\vec{k}_n$ takes the form
        \begin{equation}
            T_{\vec{k}_n}^v(z) = \begin{pmatrix}
                \sum_{j=1}^{n+1} H_j^{\vec{k}_n}z^{S_n+j-1} & * \\
                \quad & \quad \\
                \sum_{j=1}^n \left[-w_n^{-k_n}D_nH_j^{\vec{k}_{n-1}}z^{S_n+j-1} \right. & * \\
                + \left. \sum_{m=0}^{n-2}(-1)^{n-j}k_{m+2,n}\sigma_{m+1,n+1}D_{m+1}D_{n+1}^{-1}H_j^{\vec{k}_m}z^{S_{m+1}+j-1} \right]
            \end{pmatrix}
        \end{equation}
    To prove this, we will proceed by induction. First, observe that each of the 3 (local) trigonometric Lax matrices can be written in the general form
        \begin{equation}
            L_i^{k_i}(z) = \begin{pmatrix}
                w_i^{-1}z^{s_i+1}-w_iz^{s_i} & w_i^{-k_i}D_i^{-1}z^{s_i+1} \\
                w_i^{-k_i}D_iz^{s_i} & -k_iw_i^{-k_i}
            \end{pmatrix}.
        \end{equation}
    By setting $i=1$, this proves the base case. Now, assume the inductive hypothesis to be true for an arbitrary index vector $\vec{k}_n$. Then by definition of the complete mixed monodromy matrix, it follows
        \begin{equation}
            T_{\vec{k}_{n+1}}^v(z) = L_{n+1}^{k_{n+1}}(z)\cdot T_{\vec{k}_n}^v(z).
        \end{equation}
    By a direct calculation, we prove the inductive step, hence prove the claim. Furthermore, from this calculation we obtain
        \begin{equation}
        \begin{split}
            T_{\vec{k}_{n+1}}^v(z)_{11} & = \sum_{j=1}^{n+1} \left[ w_{n+1}^{-1}H_j^{\vec{k}_n}z^{S_{n+1}+j} - w_{n+1}H_j^{\vec{k}_n}z^{S_{n+1}+j-1} - \sigma_{n,n+1}D_nD_{n+1}^{-1}H_j^{\vec{k}_{n-1}}z^{S_{n+1}+j} \right. \\
            & = \left. \sum_{m=0}^{n-2}(-1)^{n-j} k_{m+2,n}\sigma_{m+1,n+1}D_{m+1}D_{n+1}^{-1}H_j^{\vec{k}_m}z^{s_{n+1}+S_{m+1}+j} \right].
        \end{split}
        \end{equation}
    By convention, $H_i^{\vec{k}_{n+1}}$ is the coefficient of $z^{S_{n+1}+i-1}$. Using this and the equation above, we obtain the recursion formula.
\end{proof}

\begin{corollary}
    Let $\vec{k}_{n+1}$ be an index vector. Then the $A_n$ q-Toda Lax Hamiltonians obey the following symmetry:
        \begin{equation}
            H_i^{\vec{k}_{n+1}} = \bar{H}_{n+2-i}^{-\vec{k}_{n+1}}
        \end{equation}
    where $\bar{H}_i^{\vec{k}_n}$ are the type $A_n$ Hamiltonians with $w_i\mapsto w_i^{-1}$.
\end{corollary}

\begin{figure}
    \begin{subfigure}[b]{1\textwidth}
        \centering
        \begin{tikzpicture}
            \tikzset{VertexStyle/.style = {shape = circle,fill = black,minimum size = 1.5mm,inner sep=0pt}}
            \Vertex[y=1,color=blue]{A2} \Vertex[x=4,y=1,color=black]{C2} \Vertex[x=9,y=1,color=orange]{E2} \Vertex[x=13,y=1,color=red]{F2}
            \Edge[Direct](A2)(C2) \Edge[Direct](C2)(E2) \Edge[Direct](E2)(F2)
            \Vertex[y=2,color=blue]{A3} \Vertex[x=3,y=2,color=orange]{B3} \Vertex[x=6,y=2,color=black]{C3} \Vertex[x=10,y=2,color=black]{D3} \Vertex[x=11,y=2,color=orange]{E3} \Vertex[x=13,y=2,color=red]{F3}
            \Edge[Direct](A3)(B3) \Edge[Direct](B3)(C3) \Edge[Direct](C3)(D3) \Edge[Direct](D3)(E3) \Edge[Direct](E3)(F3) \Edge[Direct](B3)(C2) \Edge[Direct](E2)(D3)
            \Vertex[y=3,color=blue]{A4} \Vertex[x=5,y=3,color=orange]{B4} \Vertex[x=12,y=3,color=black]{C4} \Vertex[x=13,y=3,color=red]{D4}
            \Edge[Direct](A4)(B4) \Edge[Direct](B4)(C4) \Edge[Direct](C4)(D4) \Edge[Direct](B4)(C3) \Edge[Direct](E3)(C4)
            \Text[x=13.5,y=1]{i} \Text[x=13.5,y=2]{i+1} \Text[x=13.5,y=3]{i+2}
        \end{tikzpicture}
        \caption{$Q_i=0$}
        \label{fig:9a}
    \end{subfigure}
    \par\bigskip
    \begin{subfigure}[b]{1\textwidth}
        \centering
        \begin{tikzpicture}
            \tikzset{VertexStyle/.style = {shape = circle,fill = black,minimum size = 1.5mm,inner sep=0pt}}
            \Vertex[y=1,color=blue]{A2} \Vertex[x=6,y=1,color=black]{C2} \Vertex[x=9,y=1,color=orange]{E2} \Vertex[x=13,y=1,color=red]{F2}
            \Edge[Direct](A2)(C2) \Edge[Direct](C2)(E2) \Edge[Direct](E2)(F2)
            \Vertex[y=2,color=blue]{A3} \Vertex[x=5,y=2,color=orange]{B3} \Vertex[x=4,y=2,color=black]{C3} \Vertex[x=10,y=2,color=black]{D3} \Vertex[x=11,y=2,color=orange]{E3} \Vertex[x=13,y=2,color=red]{F3}
            \Edge[Direct](A3)(C3) \Edge[Direct](C3)(B3) \Edge[Direct](B3)(D3) \Edge[Direct](D3)(E3) \Edge[Direct](E3)(F3) \Edge[Direct](B3)(C2) \Edge[Direct](E2)(D3)
            \Vertex[y=3,color=blue]{A4} \Vertex[x=3,y=3,color=orange]{B4} \Vertex[x=12,y=3,color=black]{C4} \Vertex[x=13,y=3,color=red]{D4}
            \Edge[Direct](A4)(B4) \Edge[Direct](B4)(C4) \Edge[Direct](C4)(D4) \Edge[Direct](B4)(C3) \Edge[Direct](E3)(C4)
            \Text[x=13.5,y=1]{i} \Text[x=13.5,y=2]{i+1} \Text[x=13.5,y=3]{i+2}
        \end{tikzpicture}
        \caption{$Q_i=1$}
        \label{fig:9b}
    \end{subfigure}
    \par\bigskip
    \begin{subfigure}[b]{1\textwidth}
        \centering
        \begin{tikzpicture}
            \tikzset{VertexStyle/.style = {shape = circle,fill = black,minimum size = 1.5mm,inner sep=0pt}}
            \Vertex[y=1,color=blue]{A2} \Vertex[x=4,y=1,color=black]{C2} \Vertex[x=11,y=1,color=orange]{E2} \Vertex[x=13,y=1,color=red]{F2}
            \Edge[Direct](A2)(C2) \Edge[Direct](C2)(E2) \Edge[Direct](E2)(F2)
            \Vertex[y=2,color=blue]{A3} \Vertex[x=3,y=2,color=orange]{B3} \Vertex[x=6,y=2,color=black]{C3} \Vertex[x=12,y=2,color=black]{D3} \Vertex[x=9,y=2,color=orange]{E3} \Vertex[x=13,y=2,color=red]{F3}
            \Edge[Direct](A3)(B3) \Edge[Direct](B3)(C3) \Edge[Direct](C3)(E3) \Edge[Direct](E3)(D3) \Edge[Direct](D3)(F3) \Edge[Direct](B3)(C2) \Edge[Direct](E2)(D3)
            \Vertex[y=3,color=blue]{A4} \Vertex[x=5,y=3,color=orange]{B4} \Vertex[x=10,y=3,color=black]{C4} \Vertex[x=13,y=3,color=red]{D4}
            \Edge[Direct](A4)(B4) \Edge[Direct](B4)(C4) \Edge[Direct](C4)(D4) \Edge[Direct](B4)(C3) \Edge[Direct](E3)(C4)
            \Text[x=13.5,y=1]{i} \Text[x=13.5,y=2]{i+1} \Text[x=13.5,y=3]{i+2}
        \end{tikzpicture}
        \caption{$Q_i=-1$}
        \label{fig:9c}
    \end{subfigure}
    \caption{Directed Networks associated to Quiver Blocks}
    \label{fig:9}
\end{figure}
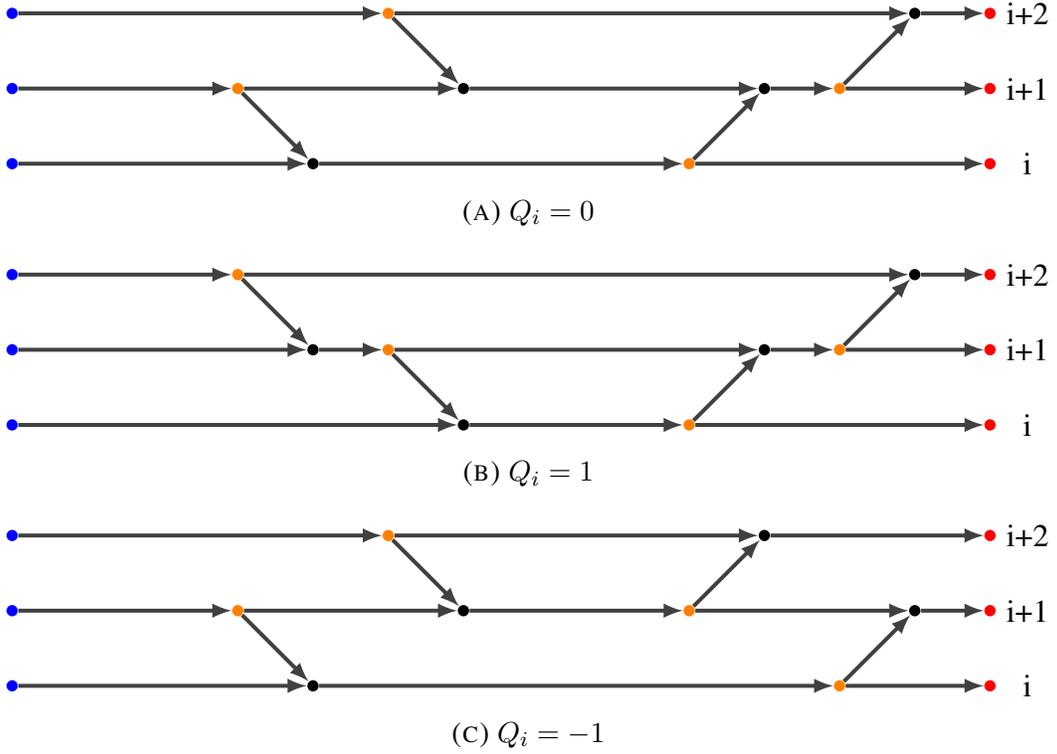

Using Remark 3.14, we can now assign explicit quantized path weights. Denote $\alpha:\mathcal{X}_\mathbf{i}^q\rightarrow\mathcal{A}_n^v$ the map such that
    \begin{equation}
        X_{i,j} \mapsto \left\{ \begin{matrix} 
            w_i^{-2} & \quad \text{if }i=j \\
            \hat{\sigma}_{i,j}D_iD_j^{-1} & \quad \text{if }i\neq j
        \end{matrix}\right., \quad i,j\in[1,n+1].
    \end{equation}
where $\hat{\sigma}_{i,j} = w_i^{-Q_i-1}w_{i+1}^{-Q_{i+1}-1}\cdots w_j^{-Q_j-1}$.

\begin{lemma}
    The map $\alpha$ is a homomorphism of $\mathbb{C}(v)$-algebras under the identification $v=q$.
\end{lemma}

\begin{proof}
    For the quantized path weights, we can compute the following relations:
        \begin{equation}
        \begin{split}
            X_{i,i}X_{j,j} & = X_{j,j}X_{i,i} \\
            X_{i,j}X_{l,l} & = \left\{ \begin{matrix}
                q^{-2}X_{l,l}X_{i,j} & \quad \text{if }i=l \\
                q^{2}X_{l,l}X_{i,j} & \quad \text{if }j=l \\
                X_{l,l}X_{i,j} & \quad \text{otherwise}
            \end{matrix} \right. \\
            X_{i,j}X_{l,m} & = q^AX_{l,m}X_{i,j}
        \end{split}
        \end{equation}
    where the exponent of $q$ in the last equation can take on the following values
        \begin{equation}
            A = \left\{ \begin{matrix}
                0 & \quad \text{if }(i=l,j=m),\text{ }(i<l,j<m),\text{ or }(i>l,j>m) \\
                2 & \quad \text{if }(i=l,j<m),\text{ or }(i>l,j=m) \\
                -2 & \quad \text{if }(i=l,j>m),\text{ or }(i<l,j=m) \\
                4 & \quad \text{if }(i>l,j<m) \\
                -4 & \quad \text{if }(i<l,j>m)
            \end{matrix} \right.
        \end{equation}
    It is a simple computation to check that the terms $w_i^{-2},\hat{\sigma}_{i,j}D_iD_j^{-1}$ satisfy the same commutation relations under the identification $v=q$.
\end{proof}

Lemma 4.4 allows us to identify the quantized path weights with the elements of $\mathcal{A}_n^v$:
    \begin{equation}
        X_{i,j} = \left\{ \begin{matrix} 
            w_i^{-2} & \quad \text{if }i=j \\
            \hat{\sigma}_{i,j}D_iD_j^{-1} & \quad \text{if }i\neq j
        \end{matrix}\right., \quad i,j\in[1,n+1].
    \end{equation}

Now, we are in a position to prove the main theorem of this section:

\begin{theorem}
    The type $A_n$ q-Toda Hamiltonians from the network formulation are related to the type $A_n$ q-Toda Hamiltonians from the Lax formulation by the formula
        \begin{equation}
            H_i^{\mathbf{i}_n} = (w_1^{-1}\cdots w_{n+1}^{-1})H_{i+1}^{(0,Q_{n-1},0)}, \quad i\in[1,n]
        \end{equation}
\end{theorem}

\begin{proof}
    We will proceed by induction on the quiver blocks. The base case is the type $A_2$ quiver $Q_\mathbf{i}=Q_1$ consisting of a single quiver block. By counting paths on the associated directed network for the 3 possible quiver blocks in \textbf{Figure 7}, we obtain
        \begin{equation}
        \begin{split}
            H_1^{(-1,-2,1,2)} & = X_{1,1} + X_{2,2} + X_{3,3} + X_{1,2} + X_{2,3} \\
            H_2^{(-1,-2,1,2)} & = X_{1,1}X_{2,2} + X_{1,1}X_{3,3} + X_{2,2}X_{3,3} + X_{1,1}X_{2,3} + X_{1,2}X_{3,3}
        \end{split}
        \end{equation}
        
        \begin{equation}
        \begin{split}
            H_1^{(-2,-1,1,2)} & = X_{1,1} + X_{2,2} + X_{3,3} + X_{1,2} + X_{2,3} + X_{1,3} \\
            H_2^{(-2,-1,1,2)} & = X_{1,1}X_{2,2} + X_{1,1}X_{3,3} + X_{2,2}X_{3,3} + X_{1,1}X_{2,3} + X_{1,2}X_{3,3}
        \end{split}
        \end{equation}
        
        \begin{equation}
        \begin{split}
            H_1^{(-1,-2,2,1)} & = X_{1,1} + X_{2,2} + X_{3,3} + X_{1,2} + X_{2,3} \\
            H_2^{(-1,-2,2,1)} & = X_{1,1}X_{2,2} + X_{1,1}X_{3,3} + X_{2,2}X_{3,3} + X_{1,1}X_{2,3} + X_{1,2}X_{3,3} \\
            & \quad + X_{1,2}X_{2,3}
        \end{split}
        \end{equation}
    On the other hand, the 3 sets of Hamiltonians obtained via Lax matrices are
        \begin{equation}
        \begin{split}
            H_2^{(0,0,0)} & = w_1^{-1}w_2w_3 + w_1w_2^{-1}w_3 + w_1w_2w_3^{-1} + w_3D_1D_2^{-1} + w_1D_2D_3^{-1} \\
            H_3^{(0,0,0)} & = w_1w_2^{-1}w_3^{-1} + w_1^{-1}w_2w_3^{-1} + w_1^{-1}w_2^{-1}w_3 + w_3^{-1}D_1D_2^{-1} + w_1^{-1}D_2D_3^{-1}
        \end{split}
        \end{equation}
        
        \begin{equation}
        \begin{split}
            H_2^{(0,1,0)} & = w_1^{-1}w_2w_3 + w_1w_2^{-1}w_3 + w_1w_2w_3^{-1} + w_2^{-1}w_3D_1D_2^{-1} + w_1w_2^{-1}D_2D_3^{-1} \\
            & \quad + w_2^{-1}D_1D_3^{-1} \\
            H_3^{(0,1,0)} & = w_1w_2^{-1}w_3^{-1} + w_1^{-1}w_2w_3^{-1} + w_1^{-1}w_2^{-1}w_3 + w_2^{-1}w_3^{-1}D_1D_2^{-1} \\
            & \quad + w_1^{-1}w_2^{-1}D_2D_3^{-1}
        \end{split}
        \end{equation}
        
        \begin{equation}
        \begin{split}
            H_2^{(0,-1,0)} & = w_1^{-1}w_2w_3 + w_1w_2^{-1}w_3 + w_1w_2w_3^{-1} + w_2w_3D_1D_2^{-1} + w_1w_2D_2D_3^{-1} \\
            H_3^{(0,-1,0)} & = w_1w_2^{-1}w_3^{-1} + w_1^{-1}w_2w_3^{-1} + w_1^{-1}w_2^{-1}w_3 + w_2w_3^{-1}D_1D_2^{-1} \\
            & \quad + w_1^{-1}w_2D_2D_3^{-1} + w_2D_1D_3^{-1}
        \end{split}
        \end{equation}
    Thus, we can see that under our homomorphism,
        \begin{equation}
        \begin{split}
            H_i^{(-1,-2,1,2)} & = (w_1^{-1}w_2^{-1}w_3^{-1})H_{i+1}^{(0,0,0)} \\
            H_i^{(-2,-1,1,2)} & = (w_1^{-1}w_2^{-1}w_3^{-1})H_{i+1}^{(0,1,0)} \\
            H_i^{(-1,-2,2,1)} & = (w_1^{-1}w_2^{-1}w_3^{-1})H_{i+1}^{(0,-1,0)}.
        \end{split}
        \end{equation}
    
    Now, assume the inductive hypothesis for the first $n-2$ blocks $Q_1,\ldots,Q_{n-2}$ for type $A_{n-1}$. To prove the inductive step, here are 3 cases to investigate corresponding to $Q_{n-1}=0,1,-1$.
    
    If $Q_{n-1}=0$, we obtain 2 new paths corresponding to the weights $X_{n,n+1}$ and $X_{n+1,n+1}$, which gives us
        \begin{equation}
            (H_i^{\mathbf{i}_n})_0 = H_i^{\mathbf{i}_{n-1}} + X_{n+1,n+1}H_{i-1}^{\mathbf{i}_{n-1}} + X_{n,n+1}H_{i-1}^{\mathbf{i}_{n-2}}.
        \end{equation}
    Using the inductive hypothesis, it follows
        \begin{equation}
        \begin{split}
            (H_i^{\mathbf{i}_n})_0 & = (w_1^{-1}\cdots w_n^{-1})H_{i+1}^{(0,\vec{Q}_{n-2},0)} + \omega_{n+1}^{-2}(w_1^{-1}\cdots w_n^{-1})H_i^{(0,\vec{Q}_{n-2},0)} + \hat{\sigma}_{n,n+1}D_nD_{n+1}^{-1}(w_1^{-1}\cdots w_{n-1}^{-1})H_i^{(0,\vec{Q}_{n-3},0)} \\
            & = (w_1^{-1}\cdots w_{n+1}^{-1}) (w_{n+1}H_{i+1}^{(0,\vec{Q}_{n-2},0)} + w_{n+1}^{-1}H_i^{(0,\vec{Q}_{n-2},0)} + \sigma_{n,n+1}D_nD_{n+1}^{-1}H_i^{(0,\vec{Q}_{n-3},0)}) \\
            & = (w_1^{-1}\cdots w_{n+1}^{-1})H_{i+1}^{(0,\vec{Q}_{n-1},0)}
        \end{split}
        \end{equation}
    
    Now, suppose $Q_{n-1}=1$. It will make things easier if we split up the quiver vector in the following way:
        \begin{equation}
            \vec{Q}_{n-1} = (\vec{Q}_{n-1,j_r+1},\vec{Q}_{j_r,j_{r-1}+1},\ldots,\vec{Q}_{j_1,j_0+1},Q_{j_0},\vec{Q}_{j_0-1})
        \end{equation}
    where $\vec{Q}_{i,j}=(Q_i,Q_{i-1},\ldots,Q_j)$ such that $\vec{Q}_{n-1,j_r+1}=(1,1,\ldots,1), \vec{Q}_{j_r,j_{r-1}+1}=(-1,-1,\ldots,-1), \ldots$ and $Q_{j_0}=0$ is the leftmost entry equal to 0. In addition to the paths $X_{n,n+1},X_{n+1,n+1}$, the subvector $\vec{Q}_{n-1,j_r+1}$ offers the paths $X_{j_r+1,n+1},\ldots,X_{n-1,n+1}$. The contribution to the $i$-th Hamiltonian corresponding to these paths is
        \begin{equation}
            (H_i^{\mathbf{i}_n})_1 = \sum_{m=j_r}^{n-2} X_{m+1,n+1}H_{i-1}^{\mathbf{i}_{m-1}}.
        \end{equation}
    Moving on to the subvector $\vec{Q}_{j_r,j_{r-1}+1}=(-1,-1,\ldots,-1)$, inspection of the directed network tells us that there are no new path contributions. However, the paths $X_{m,m+1}$ and $X_{m+1,m+2}$ no longer intersect for $m\in[j_{r-1}+1,j_r]$. Furthermore, $X_{j_r,j_r+1}$ does not intersect $X_{j_r+1,n+1}$. Therefore, all of these paths can be multiplied together to obtain terms in the Hamiltonian. Hence, for every $m\in[j_{r-1},j_r-1]$, we obtain the term
        \begin{equation}
            X_{m+1,m+2}X_{m+2,m+3}\cdots X_{j_r,j_r+1}X_{j_r+1,n+1}H_{i-1-(j_r-m)}^{\mathbf{i}_{m-1}}
        \end{equation}
    Therefore, the contribution to the $i$-th Hamiltonian from the subvector $\vec{Q}_{j_r,j_{r-1}+1}$ is
        \begin{equation}
            (H_i^{\mathbf{i}_n})_2 = \sum_{m=j_{r-1}}^{j_r-1} X_{m+1,m+2}\cdots X_{j_r,j_r+1}X_{j_r+1,n+1}H_{i-1-(j_r-m)}^{\mathbf{i}_{m-1}}
        \end{equation}
    The next subvector $\vec{Q}_{j_{r-1},j_{r-2}+1}=(1,1,\ldots1)$ gives us the paths $X_{j_{r-2}+1,j_{r-1}+2},\ldots X_{j_{r-1},j_{r-1}+2}$, which do not intersect the paths from $\vec{Q}_{j_r,j_{r-1}+1}$. Thus, we can multiply all of these paths together to obtain the contribution
        \begin{equation}
            (H_i^{\mathbf{i}_n})_3 = \sum_{m=j_{r-2}}^{j_{r-1}-1} X_{m+1,j_{r-1}+2}(X_{j_{r-1}+2,j_{r-1}+3}\cdots X_{j_r,j_r+1}X_{j_r+1,n+1})H_{i-1-(j_r-m)}^{\mathbf{i}_{m-1}}
        \end{equation}
    It now becomes clear that next contribution from $\vec{Q}_{j_{r-2},j_{r-3}+1}=(-1,1,\ldots,-1)$ is
        \begin{equation}
        \begin{split}
            (H_i^{\mathbf{i}_n})_4 & = \sum_{m=j_{r-3}}^{j_{r-2}-1} X_{m+1,m+2}\cdots X_{j_{r-2},j_{r-2}+1} (X_{j_{r-2}+1,j_{r-1}+2}) \\
            & \quad \times (X_{j_{r-1}+2,j_{r-1}+3}\cdots X_{j_r,j_r+1}X_{j_r+1,n+1})H_{i-1-(j_r-j_{r-1})-(j_{r-2}-m)}^{\mathbf{i}_{m-1}}
        \end{split}
        \end{equation}
    and so forth. Once we get to $Q_{j_0}=0$, then there are no more extra paths to consider that are not already terms in a lower Hamiltonian. 
    
    Now, recall that we assigned the quantized path weights $X_{i,j}=\hat{\sigma}_{i,j}D_iD_j^{-1}$ for $i\neq j$. Using this fact, it follows
        \begin{equation}
        \begin{split}
            \hat{\sigma}_{m+1,n+1}D_{m+1}D_{n+1}^{-1} & = X_{m+1,n+1} \\
            & = X_{m+1,m+2}\cdots X_{j_r,j_r+1}X_{j_r+1,n+1} \\
            & = X_{m+1,j_{r-1}+2}(X_{j_{r-1}+2,j_{r-1}+3}\cdots X_{j_r,j_r+1}X_{j_r+1,n+1}) \\
            & = X_{m+1,m+2}\cdots X_{j_{r-2},j_{r-2}+1} (X_{j_{r-2}+1,j_{r-1}+2}) \\
            & \quad \times (X_{j_{r-1}+2,j_{r-1}+3}\cdots X_{j_r,j_r+1}X_{j_r+1,n+1})
        \end{split}
        \end{equation}
    So, we see that all of the coefficients in front of the Hamiltonians are equal to the weight $\hat{\sigma}_{m+1,n+1}D_{m+1}D_{n+1}^{-1}$. Thus, we have
        \begin{equation}
            H_i^{\mathbf{i}_n} = (H_i^{\mathbf{i}_n})_0 + \sum_{m=j_0}^{n-2} \hat{\sigma}_{m+1,n+1}D_{m+1}D_{n+1}^{-1} H^{\mathbf{i}_{m-1}}_{i-1-A_{r,m}}
        \end{equation}
    where $A_{r,m}=(j_r-j_{r-1})+(j_{r-2}-j_{r-3})+\cdots+(j_{r-N}-m)$ for some $1\leq N\leq r$. By construction, $A_{r,m}$ is equal to the number of entries in the subvector $\vec{Q}_{n-1,m+1}$ that are equal to $-1$. By identifying $\vec{k}_n=(0,\vec{Q}_{n-1},0)$, we can see that $S_{n,m+1}=A_{r,m}$. Lastly, notice that the leftmost entry in $\vec{Q}$ equal to 0 is the starting point for the sum. Hence, we can write this as a sum from $m=0$ and insert $k_{m+2,n}$ into the summand. Therefore, it follows
        \begin{equation}
        \begin{split}
            H_i^{\mathbf{i}_n} & = (H_i^{\mathbf{i}_n})_0 + \sum_{m=0}^{n-2} k_{m+2,n} \hat{\sigma}_{m+1,n+1}D_{m+1}D_{n+1}^{-1} H^{\mathbf{i}_{m-1}}_{i-1-S_{n,m+1}} \\
            & = (w_1^{-1}\cdots w_{n+1}^{-1}) \left( (H_{i+1}^{(0,\vec{Q}_{n-1},0)})_0 + \sum_{m=0}^{n-2} k_{m+2,n} \sigma_{m+1,n+1}D_{m+1}D_{n+1}^{-1} H^{(0,\vec{Q}_{m-2},0)}_{i-S_{n,m+1}} \right) \\
            & = (w_1^{-1}\cdots w_{n+1}^{-1}) H^{(0,\vec{Q}_{n-1},0)}_{i+1}.
        \end{split}
        \end{equation}
        
    The case where $Q_{n-1}=-1$ is treated analogously.
\end{proof}

\subsection{Type C}

Using the Lax formalism for $C_n$ in \cite{GT19}, we provide a recursive formula for the type $C_n$ q-Toda Hamiltonians in terms of type $A$ q-Toda Hamiltonians. Similarly to the previous section, we will use this formula to prove a bijection between the q-Toda Hamiltonians obtained via Lax matrices vs. quantized weights on the type C directed network.

\begin{theorem}
    Let $\vec{k}_n\in\{-1,0,1\}^n$ be an index vector for the \textit{double complete mixed monodromy matrix} $\mathbb{T}_{\vec{k}_n}^v(z)$. Then the $i$-th type $C_n$ q-Toda Hamiltonian associated to $\vec{k}_n$ can be written as
        \begin{equation}
        \begin{split}
            \mathbb{H}_i^{\vec{k}_n} & = \sum_{j=1}^{n+1} \left[ H_{n+1+j-i}^{\vec{k}_n}H_j^{\vec{k}_n} + \omega_n^{-2k_n}D_n^2H_{n+1+j-i}^{\vec{k}_{n-1}}H_j^{\vec{k}_{n-1}} \right] + \sum_{j=1}^{n+1}\sum_{m=0}^{n-2} k_{m+2,n}\sigma_{m+1,n} \\
            & \quad \times \left[ H_{n+1+j-i}^{\vec{k}_{n-1}}(\omega_n^{-k_n}D_{m+1}D_n)H_{j-S_{n,m+1}}^{\vec{k}_m} + H_{n+1+j-i+S_{n,m+1}}^{\vec{k}_m}(\omega^{-k_n}D_{m+1}D_n)H_j^{\vec{k}_{n-1}} \right] \\
            & \quad + \sum_{j=1}^{n+1}\sum_{m,m'=0}^{n-2} k_{m+2,n}k_{m'+2,n}\sigma_{m+1,n}\sigma_{m'+1,n} H_{j+i-n-1+S_{m'+1,m+1}}^{\vec{k}_{m'}}(\omega_n^{-k_n}D_{m'+1}D_{m+1}D_n)H^{\vec{k}_m}_j
        \end{split}
        \end{equation}
\end{theorem}

\begin{proof}
    Recall the 3 (local) Lax matrices defined in \cite{GT19}, $\bar{L}_i^{-k_i}(z)$. Observe that these matrices are related to the original matrices $L_i^{k_i}(z)$ in the following way:
        \begin{equation}
            \bar{L}_i^{-k_i}(z) = -[L_i^{k_i}(z^{-1})]^T.
        \end{equation}
    Hence, the \textit{double complete mixed monodromy matrix} can be expressed as
        \begin{equation}
        \begin{split}
            \mathbb{T}_{\vec{k}_n}^v(z) & = \bar{L}_1^{v,-k_1}(z)\cdots\bar{L}_n^{v,-k_n}(z)L_n^{v,k_n}(z)\cdots L_1^{v,k_1}(z) \\
            & =  (-1)^n[L_1^{k_1}(z^{-1})]^T\cdots[L_n^{k_n}(z^{-1})]^T \cdot T_{\vec{k}_n}^v(z) \\
            & = (-1)^n[L_n^{k_n}(z^{-1})\cdots L_1^{k_1}(z^{-1})]^T \cdot T_{\vec{k}_n}^v(z) \\
            & = (-1)^n[T_{\vec{k}_n}^v(z^{-1})]^T \cdot T_{\vec{k}_n}^v(z)
        \end{split}
        \end{equation}
    where $T_{\vec{k}_n}^v(z)$ is the type A \textit{complete mixed monodromy matrix}. By direct multiplication and using the convention that $\mathbb{H}_i^{\vec{k}_n}$ is the coefficient of $z^{-n+i-1}$, the formula follows.
\end{proof}

\begin{corollary}
    Let $\vec{k}_n=(k_n,\ldots,k_1)$ be an index vector. Suppose $\vec{k}_n'=(k_n,\ldots,k_2,0)$ is another index vector. Then it follows
        \begin{equation}
            \mathbb{H}_i^{\vec{k}_n} = \mathbb{H}_i^{\vec{k}_n'}.
        \end{equation}
\end{corollary}

Now, recall the 3 possible quiver blocks $Q_i=0,1,-1$ from the previous section. These possible blocks are the same for type $C_n$ with the exception the top block $Q_{n-1}$ has double the amount of arrows, as illustrated in \textbf{Figure 9}. The directed subnetwork corresponding to the different quiver blocks is given in \textbf{Figure 10}. Furthermore, in \textbf{Figure 10} we give show the directed networks corresponding to the top block $Q_{n-1}$.

\begin{figure}
    \begin{subfigure}[b]{1\textwidth}
        \centering
        \begin{tikzpicture}
            \tikzset{VertexStyle/.style = {shape = circle,fill = black,minimum size = 1.5mm,inner sep=0pt}}
            \Vertex[y=1,color=blue]{A2} \Vertex[x=4,y=1,color=black]{C2} \Vertex[x=9,y=1,color=orange]{E2} \Vertex[x=13,y=1,color=red]{F2}
            \Edge[Direct](A2)(C2) \Edge[Direct](C2)(E2) \Edge[Direct](E2)(F2)
            \Vertex[y=2,color=blue]{A3} \Vertex[x=3,y=2,color=orange]{B3} \Vertex[x=6,y=2,color=black]{C3} \Vertex[x=10,y=2,color=black]{D3} \Vertex[x=11,y=2,color=orange]{E3} \Vertex[x=13,y=2,color=red]{F3}
            \Edge[Direct](A3)(B3) \Edge[Direct](B3)(C3) \Edge[Direct](C3)(D3) \Edge[Direct](D3)(E3) \Edge[Direct](E3)(F3) \Edge[Direct](B3)(C2) \Edge[Direct](E2)(D3)
            \Vertex[y=3,color=blue]{A4} \Vertex[x=5,y=3,color=orange]{B4} \Vertex[x=12,y=3,color=black]{C4} \Vertex[x=13,y=3,color=red]{D4}
            \Edge[Direct](A4)(B4) \Edge[Direct](B4)(C4) \Edge[Direct](C4)(D4) \Edge[Direct](B4)(C3) \Edge[Direct](E3)(C4)
            \Vertex[x=7.5,y=3.25]{X1}
            \Vertex[x=7.5,y=3.5]{X2}
            \Vertex[x=7.5,y=3.75]{X3}
            \Vertex[y=4,color=blue]{A2} \Vertex[x=6,y=4,color=black]{B2} \Vertex[x=11,y=4,color=orange]{C2} \Vertex[x=13,y=4,color=red]{D2}
            \Edge[Direct](A2)(B2) \Edge[Direct](B2)(C2) \Edge[Direct](C2)(D2)
            \Vertex[y=5,color=blue]{A3} \Vertex[x=4,y=5,color=black]{B3} \Vertex[x=5,y=5,color=orange]{C3} \Vertex[x=9,y=5,color=orange]{D3} \Vertex[x=12,y=5,color=black]{E3} \Vertex[x=13,y=5,color=red]{F3}
            \Edge[Direct](A3)(B3) \Edge[Direct](B3)(C3) \Edge[Direct](C3)(D3) \Edge[Direct](D3)(E3) \Edge[Direct](E3)(F3) \Edge[Direct](C3)(B2) \Edge[Direct](C2)(E3)
            \Vertex[y=6,color=blue]{A4} \Vertex[x=3,y=6,color=orange]{B4} \Vertex[x=10,y=6,color=black]{C4} \Vertex[x=13,y=6,color=red]{D4}
            \Edge[Direct](A4)(B4) \Edge[Direct](B4)(C4) \Edge[Direct](C4)(D4) \Edge[Direct](B4)(B3) \Edge[Direct](D3)(C4)
            \Text[x=14,y=1]{i} \Text[x=14,y=2]{i+1} \Text[x=14,y=3]{i+2}
            \Text[x=14,y=4]{2n-1-i} \Text[x=14,y=5]{2n-i} \Text[x=14,y=6]{2n+1-i}
        \end{tikzpicture}
        \caption{$Q_i=0$}
        \label{fig:10a}
    \end{subfigure}
    \par\bigskip
    \begin{subfigure}[b]{1\textwidth}
        \centering
        \begin{tikzpicture}
            \tikzset{VertexStyle/.style = {shape = circle,fill = black,minimum size = 1.5mm,inner sep=0pt}}
            \Vertex[y=1,color=blue]{A2} \Vertex[x=6,y=1,color=black]{C2} \Vertex[x=9,y=1,color=orange]{E2} \Vertex[x=13,y=1,color=red]{F2}
            \Edge[Direct](A2)(C2) \Edge[Direct](C2)(E2) \Edge[Direct](E2)(F2)
            \Vertex[y=2,color=blue]{A3} \Vertex[x=5,y=2,color=orange]{B3} \Vertex[x=4,y=2,color=black]{C3} \Vertex[x=10,y=2,color=black]{D3} \Vertex[x=11,y=2,color=orange]{E3} \Vertex[x=13,y=2,color=red]{F3}
            \Edge[Direct](A3)(C3) \Edge[Direct](C3)(B3) \Edge[Direct](B3)(D3) \Edge[Direct](D3)(E3) \Edge[Direct](E3)(F3) \Edge[Direct](B3)(C2) \Edge[Direct](E2)(D3)
            \Vertex[y=3,color=blue]{A4} \Vertex[x=3,y=3,color=orange]{B4} \Vertex[x=12,y=3,color=black]{C4} \Vertex[x=13,y=3,color=red]{D4}
            \Edge[Direct](A4)(B4) \Edge[Direct](B4)(C4) \Edge[Direct](C4)(D4) \Edge[Direct](B4)(C3) \Edge[Direct](E3)(C4)
            \Vertex[x=7.5,y=3.25]{X1}
            \Vertex[x=7.5,y=3.5]{X2}
            \Vertex[x=7.5,y=3.75]{X3}
            \Vertex[y=4,color=blue]{A2} \Vertex[x=4,y=4,color=black]{C2} \Vertex[x=11,y=4,color=orange]{E2} \Vertex[x=13,y=4,color=red]{F2}
            \Edge[Direct](A2)(C2) \Edge[Direct](C2)(E2) \Edge[Direct](E2)(F2)
            \Vertex[y=5,color=blue]{A3} \Vertex[x=3,y=5,color=orange]{B3} \Vertex[x=6,y=5,color=black]{C3} \Vertex[x=12,y=5,color=black]{D3} \Vertex[x=9,y=5,color=orange]{E3} \Vertex[x=13,y=5,color=red]{F3}
            \Edge[Direct](A3)(B3) \Edge[Direct](B3)(C3) \Edge[Direct](C3)(E3) \Edge[Direct](E3)(D3) \Edge[Direct](D3)(F3) \Edge[Direct](B3)(C2) \Edge[Direct](E2)(D3)
            \Vertex[y=6,color=blue]{A4} \Vertex[x=5,y=6,color=orange]{B4} \Vertex[x=10,y=6,color=black]{C4} \Vertex[x=13,y=6,color=red]{D4}
            \Edge[Direct](A4)(B4) \Edge[Direct](B4)(C4) \Edge[Direct](C4)(D4) \Edge[Direct](B4)(C3) \Edge[Direct](E3)(C4)
            \Text[x=14,y=1]{i} \Text[x=14,y=2]{i+1} \Text[x=14,y=3]{i+2}
            \Text[x=14,y=4]{2n-1-i} \Text[x=14,y=5]{2n-i} \Text[x=14,y=6]{2n+1-i}
        \end{tikzpicture}
        \caption{$Q_i=1$}
        \label{fig:10b}
    \end{subfigure}
    \par\bigskip
    \begin{subfigure}[b]{1\textwidth}
        \centering
        \begin{tikzpicture}
            \tikzset{VertexStyle/.style = {shape = circle,fill = black,minimum size = 1.5mm,inner sep=0pt}}
            \Vertex[y=1,color=blue]{A2} \Vertex[x=4,y=1,color=black]{C2} \Vertex[x=11,y=1,color=orange]{E2} \Vertex[x=13,y=1,color=red]{F2}
            \Edge[Direct](A2)(C2) \Edge[Direct](C2)(E2) \Edge[Direct](E2)(F2)
            \Vertex[y=2,color=blue]{A3} \Vertex[x=3,y=2,color=orange]{B3} \Vertex[x=6,y=2,color=black]{C3} \Vertex[x=12,y=2,color=black]{D3} \Vertex[x=9,y=2,color=orange]{E3} \Vertex[x=13,y=2,color=red]{F3}
            \Edge[Direct](A3)(B3) \Edge[Direct](B3)(C3) \Edge[Direct](C3)(E3) \Edge[Direct](E3)(D3) \Edge[Direct](D3)(F3) \Edge[Direct](B3)(C2) \Edge[Direct](E2)(D3)
            \Vertex[y=3,color=blue]{A4} \Vertex[x=5,y=3,color=orange]{B4} \Vertex[x=10,y=3,color=black]{C4} \Vertex[x=13,y=3,color=red]{D4}
            \Edge[Direct](A4)(B4) \Edge[Direct](B4)(C4) \Edge[Direct](C4)(D4) \Edge[Direct](B4)(C3) \Edge[Direct](E3)(C4)
            \Vertex[x=7.5,y=3.25]{X1}
            \Vertex[x=7.5,y=3.5]{X2}
            \Vertex[x=7.5,y=3.75]{X3}
            \Vertex[y=4,color=blue]{A2} \Vertex[x=6,y=4,color=black]{C2} \Vertex[x=9,y=4,color=orange]{E2} \Vertex[x=13,y=4,color=red]{F2}
            \Edge[Direct](A2)(C2) \Edge[Direct](C2)(E2) \Edge[Direct](E2)(F2)
            \Vertex[y=5,color=blue]{A3} \Vertex[x=5,y=5,color=orange]{B3} \Vertex[x=4,y=5,color=black]{C3} \Vertex[x=10,y=5,color=black]{D3} \Vertex[x=11,y=5,color=orange]{E3} \Vertex[x=13,y=5,color=red]{F3}
            \Edge[Direct](A3)(C3) \Edge[Direct](C3)(B3) \Edge[Direct](B3)(D3) \Edge[Direct](D3)(E3) \Edge[Direct](E3)(F3) \Edge[Direct](B3)(C2) \Edge[Direct](E2)(D3)
            \Vertex[y=6,color=blue]{A4} \Vertex[x=3,y=6,color=orange]{B4} \Vertex[x=12,y=6,color=black]{C4} \Vertex[x=13,y=6,color=red]{D4}
            \Edge[Direct](A4)(B4) \Edge[Direct](B4)(C4) \Edge[Direct](C4)(D4) \Edge[Direct](B4)(C3) \Edge[Direct](E3)(C4)
            \Text[x=14,y=1]{i} \Text[x=14,y=2]{i+1} \Text[x=14,y=3]{i+2}
            \Text[x=14,y=4]{2n-1-i} \Text[x=14,y=5]{2n-i} \Text[x=14,y=6]{2n+1-i}
        \end{tikzpicture}
        \caption{$Q_i=-1$}
        \label{fig:10c}
    \end{subfigure}
    \caption{Directed Networks associated to Quiver Blocks}
    \label{fig:10}
\end{figure}
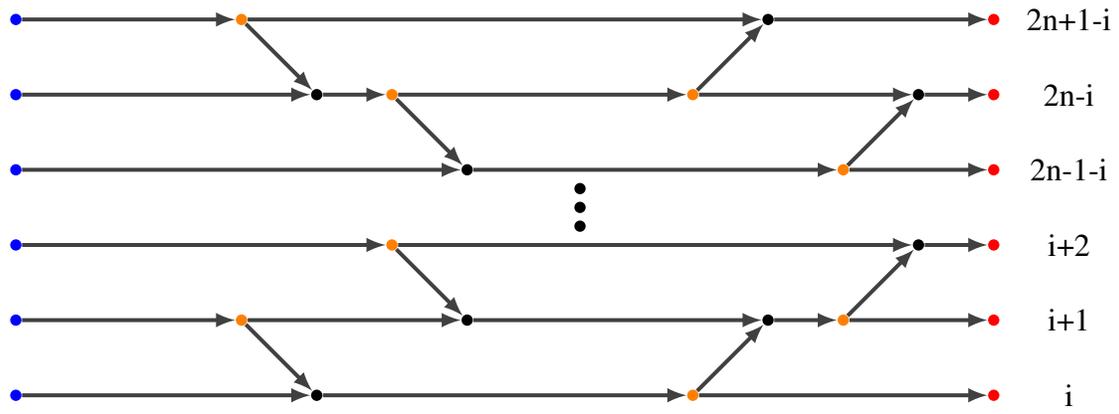
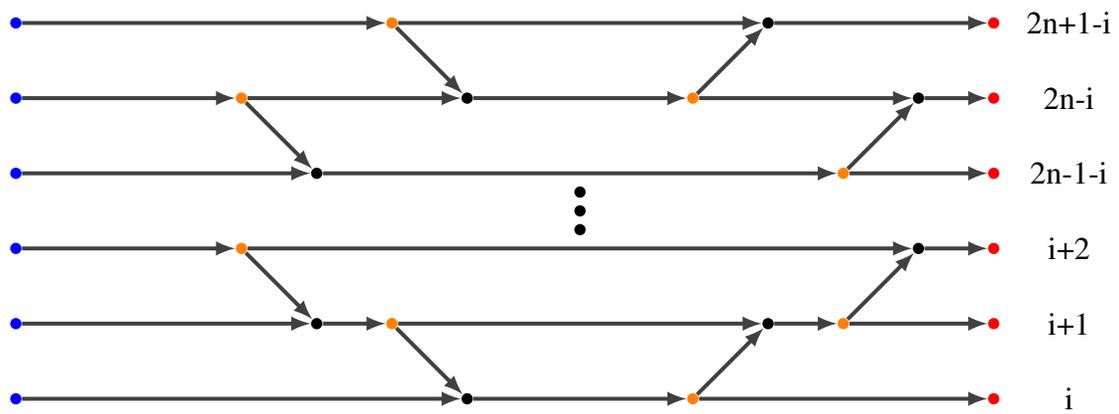
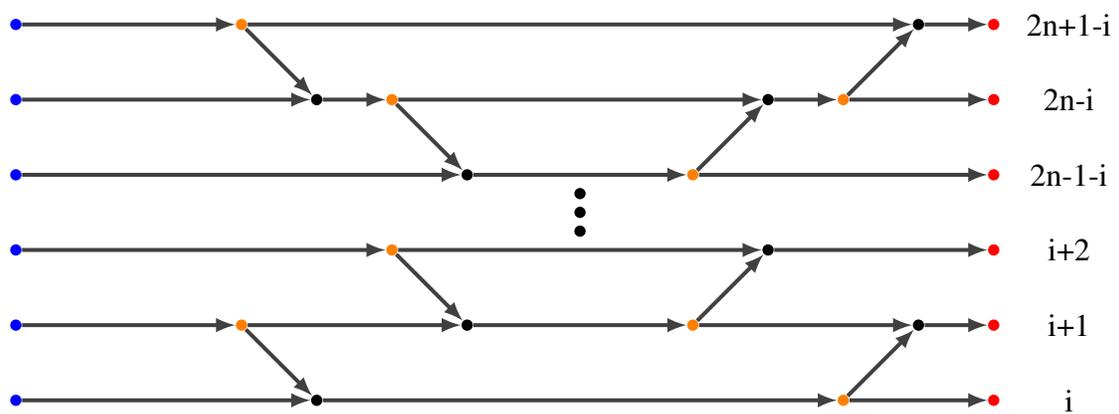

\begin{figure}
    \begin{subfigure}[b]{1\textwidth}
        \centering
        \begin{tikzpicture}
            \tikzset{VertexStyle/.style = {shape = circle,fill = black,minimum size = 1.5mm,inner sep=0pt}}
            \Vertex[y=1,color=blue]{A2} \Vertex[x=4,y=1,color=black]{C2} \Vertex[x=9,y=1,color=orange]{E2} \Vertex[x=13,y=1,color=red]{F2}
            \Edge[Direct](A2)(C2) \Edge[Direct](C2)(E2) \Edge[Direct](E2)(F2)
            \Vertex[y=2,color=blue]{A3} \Vertex[x=3,y=2,color=orange]{B3} \Vertex[x=6,y=2,color=black]{C3} \Vertex[x=10,y=2,color=black]{D3} \Vertex[x=11,y=2,color=orange]{E3} \Vertex[x=13,y=2,color=red]{F3}
            \Edge[Direct](A3)(B3) \Edge[Direct](B3)(C3) \Edge[Direct](C3)(D3) \Edge[Direct](D3)(E3) \Edge[Direct](E3)(F3) \Edge[Direct](B3)(C2) \Edge[Direct](E2)(D3)
            \Vertex[y=3,color=blue]{A4} \Vertex[x=5,y=3,color=orange]{B4} \Vertex[x=12,y=3,color=black]{C4} \Vertex[x=13,y=3,color=red]{D4}
            \Vertex[x=4,y=3,color=black]{X1} \Vertex[x=9,y=3,color=orange]{X2}
            \Edge[Direct](A4)(X1) \Edge[Direct](X1)(B4) \Edge[Direct](B4)(X2) \Edge[Direct](X2)(C4) \Edge[Direct](C4)(D4) \Edge[Direct](B4)(C3) \Edge[Direct](E3)(C4)
            \Vertex[y=4,color=blue]{A2} \Vertex[x=3,y=4,color=orange]{C2} \Vertex[x=10,y=4,color=black]{E2} \Vertex[x=13,y=4,color=red]{F2}
            \Edge[Direct](A2)(C2) \Edge[Direct](C2)(E2) \Edge[Direct](E2)(F2) \Edge[Direct](C2)(X1) \Edge[Direct](X2)(E2) 
            \Text[x=13.5,y=1]{n-1} \Text[x=13.5,y=2]{n} \Text[x=13.5,y=3]{n+1} \Text[x=13.5,y=4]{n+2}
        \end{tikzpicture}
        \caption{$Q_{n-1}=0$}
        \label{fig:11a}
    \end{subfigure}
    \par\bigskip
    \begin{subfigure}[b]{1\textwidth}
        \centering
        \begin{tikzpicture}
            \tikzset{VertexStyle/.style = {shape = circle,fill = black,minimum size = 1.5mm,inner sep=0pt}}
            \Vertex[y=1,color=blue]{A2} \Vertex[x=6,y=1,color=black]{C2} \Vertex[x=9,y=1,color=orange]{E2} \Vertex[x=13,y=1,color=red]{F2}
            \Edge[Direct](A2)(C2) \Edge[Direct](C2)(E2) \Edge[Direct](E2)(F2)
            \Vertex[y=2,color=blue]{A3} \Vertex[x=5,y=2,color=orange]{B3} \Vertex[x=4,y=2,color=black]{C3} \Vertex[x=10,y=2,color=black]{D3} \Vertex[x=11,y=2,color=orange]{E3} \Vertex[x=13,y=2,color=red]{F3}
            \Edge[Direct](A3)(C3) \Edge[Direct](C3)(B3) \Edge[Direct](B3)(D3) \Edge[Direct](D3)(E3) \Edge[Direct](E3)(F3) \Edge[Direct](B3)(C2) \Edge[Direct](E2)(D3)
            \Vertex[y=3,color=blue]{A4} \Vertex[x=3,y=3,color=orange]{B4} \Vertex[x=12,y=3,color=black]{C4} \Vertex[x=13,y=3,color=red]{D4}
            \Vertex[x=6,y=3,color=black]{X1} \Vertex[x=9,y=3,color=orange]{X2}
            \Edge[Direct](A4)(B4) \Edge[Direct](B4)(C4) \Edge[Direct](C4)(D4) \Edge[Direct](B4)(C3) \Edge[Direct](E3)(C4)
            \Vertex[y=4,color=blue]{A2} \Vertex[x=5,y=4,color=orange]{C2} \Vertex[x=10,y=4,color=black]{E2} \Vertex[x=13,y=4,color=red]{F2}
            \Edge[Direct](A2)(C2) \Edge[Direct](C2)(E2) \Edge[Direct](E2)(F2)
            \Edge[Direct](C2)(X1) \Edge[Direct](X2)(E2)
            \Text[x=13.5,y=1]{n-1} \Text[x=13.5,y=2]{n} \Text[x=13.5,y=3]{n+1} \Text[x=13.5,y=4]{n+2}
        \end{tikzpicture}
        \caption{$Q_{n-1}=1$}
        \label{fig:11b}
    \end{subfigure}
    \par\bigskip
    \begin{subfigure}[b]{1\textwidth}
        \centering
        \begin{tikzpicture}
            \tikzset{VertexStyle/.style = {shape = circle,fill = black,minimum size = 1.5mm,inner sep=0pt}}
            \Vertex[y=1,color=blue]{A2} \Vertex[x=4,y=1,color=black]{C2} \Vertex[x=11,y=1,color=orange]{E2} \Vertex[x=13,y=1,color=red]{F2}
            \Edge[Direct](A2)(C2) \Edge[Direct](C2)(E2) \Edge[Direct](E2)(F2)
            \Vertex[y=2,color=blue]{A3} \Vertex[x=3,y=2,color=orange]{B3} \Vertex[x=6,y=2,color=black]{C3} \Vertex[x=12,y=2,color=black]{D3} \Vertex[x=9,y=2,color=orange]{E3} \Vertex[x=13,y=2,color=red]{F3}
            \Edge[Direct](A3)(B3) \Edge[Direct](B3)(C3) \Edge[Direct](C3)(E3) \Edge[Direct](E3)(D3) \Edge[Direct](D3)(F3) \Edge[Direct](B3)(C2) \Edge[Direct](E2)(D3)
            \Vertex[y=3,color=blue]{A4} \Vertex[x=5,y=3,color=orange]{B4} \Vertex[x=10,y=3,color=black]{C4} \Vertex[x=13,y=3,color=red]{D4}
            \Vertex[x=4,y=3,color=black]{X1} \Vertex[x=11,y=3,color=orange]{X2}
            \Edge[Direct](A4)(B4) \Edge[Direct](B4)(C4) \Edge[Direct](C4)(D4) \Edge[Direct](B4)(C3) \Edge[Direct](E3)(C4)
            \Vertex[y=4,color=blue]{A2} \Vertex[x=3,y=4,color=orange]{C2} \Vertex[x=12,y=4,color=black]{E2} \Vertex[x=13,y=4,color=red]{F2}
            \Edge[Direct](A2)(C2) \Edge[Direct](C2)(E2) \Edge[Direct](E2)(F2)
            \Edge[Direct](C2)(X1) \Edge[Direct](X2)(E2)
            \Text[x=13.5,y=1]{n-1} \Text[x=13.5,y=2]{n} \Text[x=13.5,y=3]{n+1} \Text[x=13.5,y=4]{n+2}
        \end{tikzpicture}
        \caption{$Q_{n-1}=-1$}
        \label{fig:11c}
    \end{subfigure}
    \caption{Directed Networks associated to $Q_{n-1}$}
    \label{fig:11}
\end{figure}
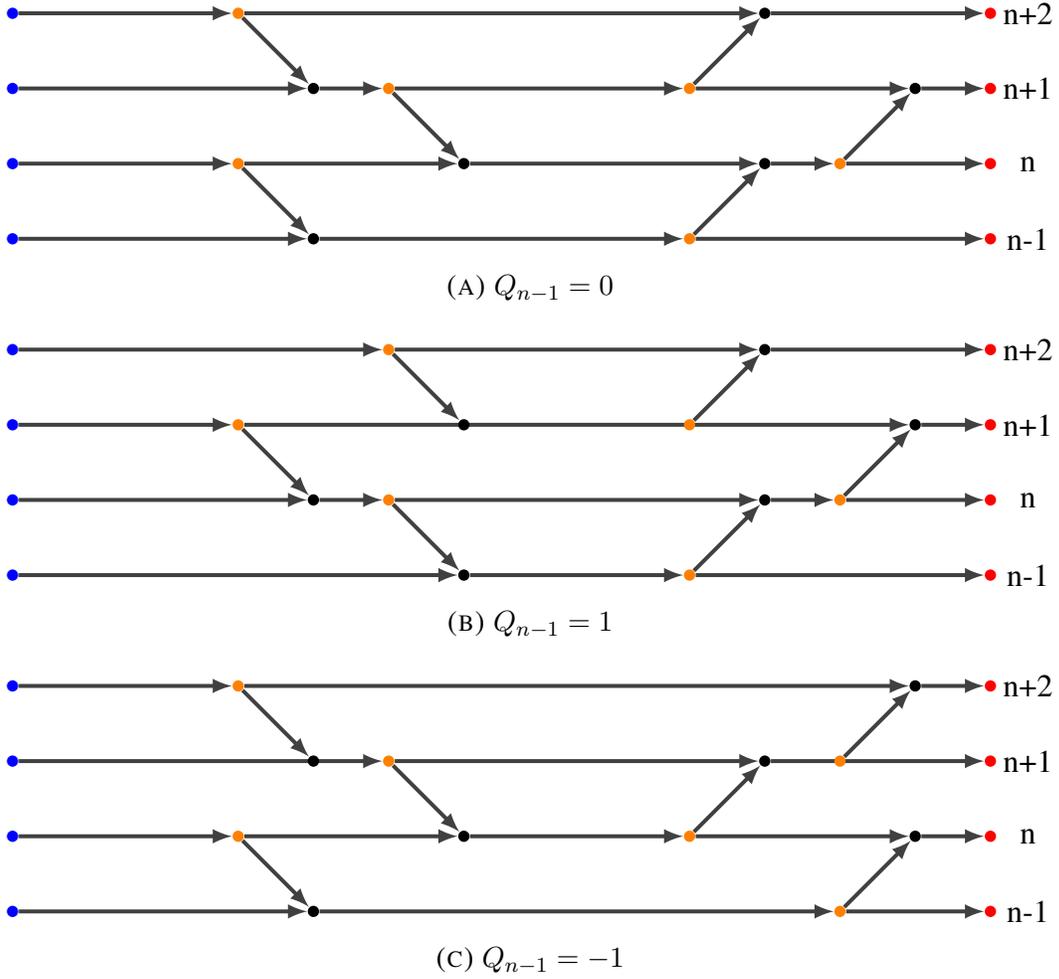

\begin{definition}
    The quantized weights for the possible types of paths in a type $C_n$ directed network are
        \begin{equation}
            X_{i,j} = \left\{ \begin{matrix}
                \omega_i^{-2} & \quad \text{if }i=j\text{ and }i\leq n \\
                \omega_{2n+1-i}^{2} & \quad \text{if }i=j\text{ and }i\geq n+1 \\
                \hat{\sigma}_{i,j} D_iD_j^{-1} & \quad \text{if }i\neq j\text{ and }j\leq n \\
                v^{-1}\hat{\sigma}_{i,n} D_iD_n & \quad \text{if }i\neq j\text{ and }j=n+1 \\
                \tilde{\sigma}_{2n+1-j,2n+1-i} D_{2n+1-j}D_{2n+1-i}^{-1} & \quad \text{if }i\neq j\text{ and }i\geq n+1 \\
                v\tilde{\sigma}_{2n+1-j,n} D_{2n+1-j}D_n & \quad \text{if }i\neq j\text{ and }i=n \\
                v^{-Q_{n-1}}w_n^{-2Q_{n-1}}D_n^2 & \quad \text{if }i=n\text{ and }j=n+1
            \end{matrix} \right.
        \end{equation}
    where $\tilde{\sigma}_{i,j} = w_i^{-Q_i+1}w_{i+1}^{-Q_{i+1}+1}\cdots w_j^{-Q_j+1}$.
\end{definition}

\begin{lemma}
    Suppose $N_{u,v}(\mathbf{i}_n)$ is a type $C_n$ directed network for an unmixed double Coxeter word $\mathbf{i}_n$ associated to $(u,v)\in W\times W$. Let $N_{u,v}(\mathbf{i})_{i,j}$ be the subnetwork consisting of rows $i$ to $j$ where $i\leq j$, and let $(\mathbb{H}_i^{\mathbf{i}_n})_{i,j}$ be the corresponding Hamiltonians. Then
        \begin{equation}
        \begin{split}
            (\mathbb{H}_i^{\mathbf{i}_n})_{1,m} & = (w_1^{-1}\cdots w_{m}^{-1})H_{i+1}^{(0,\vec{Q}_{m-2},0)}, \quad m\in[2,n] \\
            (\mathbb{H}_i^{\mathbf{i}_n})_{m',2n} & = (w_1\cdots w_{2n+1-m'})H_{n+1-i}^{(0,\vec{Q}_{2n-1-m'},0)}, \quad m'\in[n+1,2n-1]
        \end{split}
        \end{equation}
    where $H_i^{\vec{k}_m}$ are the type $A_m$ Lax Hamiltonians associated to $\vec{k}_m$ for $m\leq n$, and $\bar{H}_i^{\vec{k}_m}$ are the same Hamiltonians with $w_i\mapsto w_i^{-1}$.
\end{lemma}

\begin{theorem}
    The type $C_n$ q-Toda Hamiltonians from the network formulation with path weights given by \textbf{Definition 5.7} are related to the type $C_n$ q-Toda Hamiltonians from the Lax formulation in \cite{GT19} by the formula
        \begin{equation}
            \mathbb{H}_i^{\mathbf{i}_n} = \mathbb{H}_{i+1}^{(Q_{n-1},0)}, \quad i\in[1,n]
        \end{equation}
\end{theorem}

\begin{proof}
    First, consider the case $Q_{n-1}=0$. Then we can obtain terms by multiplying Hamiltonians from the top $n$ rows with Hamiltonians from the bottom $n$ rows, i.e.,
        \begin{equation}
            (\mathbb{H}_{i-j}^{\mathbf{i}_n})_{n+1,2n}(\mathbb{H}^{\mathbf{i}_n}_{j})_{1,n}, \quad j\in[0,i]
        \end{equation}
    We can also obtain terms by multiplying the Hamiltonians from the top $n-1$ rows with $X_{n,n+1}$ and the bottom $n-1$ rows, i.e.,
        \begin{equation}
            (\mathbb{H}_{i-j}^{\mathbf{i}_n})_{n+2,2n}X_{n,n+1}(\mathbb{H}^{\mathbf{i}_n}_{j})_{1,n-1}, \quad j\in[0,i]
        \end{equation}
    Therefore,
        \begin{equation}
        \begin{split}
            (\mathbb{H}_i^{\mathbf{i}_n})_0 & = \sum_{j=0}^{i} \left[ (\mathbb{H}_{i-j}^{\mathbf{i}_n})_{n+1,2n}(\mathbb{H}^{\mathbf{i}_n}_{j})_{1,n} + (\mathbb{H}_{i-j}^{\mathbf{i}_n})_{n+2,2n}X_{n,n+1}(\mathbb{H}^{\mathbf{i}_n}_{j})_{1,n-1} \right] \\
            & = \sum_{j=0}^i \left[ H_{n+1+j-i}^{(0,\vec{Q}_{n-2},0)}H_{j+1}^{(0,\vec{Q}_{n-2},0)} + H_{n+j-i}^{(0,\vec{Q}_{n-3},0)}(w_n^{-2Q_{n-1}}D_n^2)H_{j+1}^{(0,\vec{Q}_{n-3},0)} \right] \\
            & = (\mathbb{H}_{i+1}^{(\vec{Q}_{n-1},0)})_0
        \end{split}
        \end{equation}
    where in the second equality we used \textbf{Lemma 5.8}.
    
    Again, using \textbf{Lemma 5.8} and \textbf{Theorem 5.1}, it follows the bottom $n$ rows give
        \begin{equation}
        \begin{split}
            (\mathbb{H}_j^{\mathbf{i}_n})_{1,n} & = (w_1^{-1}\cdots w_{n}^{-1}) H_{j+1}^{(0,\vec{Q}_{n-2},0)} \\ 
            & = (w_1^{-1}\cdots w_{n}^{-1}) \left( w_{n}H_{j+1}^{(0,\vec{Q}_{n-3},0)} + w_{n+1}^{-1}H_{j}^{(0,\vec{Q}_{n-3},0)} + \sigma_{n-1,n}D_{n-1}D_{n}^{-1}H_j^{(0,\vec{Q}_{n-4},0)} \right. \\ 
            & \quad + \left. \sum_{m=0}^{n-3} k_{m+2,n}\sigma_{m+1,n}D_{m+1}D_{n}^{-1}H_{j-S_{n-1,m+1}}^{(0,\vec{Q}_{m-2},0)} \right)
        \end{split}
        \end{equation}
    and the top $n$ rows give us
        \begin{equation}
        \begin{split}
            (\mathbb{H}_j^{\mathbf{i}_n})_{n+1,2n} & = (w_1\cdots w_{n}) H_{n+1-j}^{(0,\vec{Q}_{n-2},0)} \\ 
            & = (w_1\cdots w_{n}) \left( w_{n}H_{n+2-j}^{(0,\vec{Q}_{n-3},0)} + w_{n+1}^{-1}H_{n+1-j}^{(0,\vec{Q}_{n-3},0)} + \sigma_{n-1,n}D_{n-1}D_{n}^{-1}H_{n+1-j}^{(0,\vec{Q}_{n-4},0)} \right. \\ 
            & \quad + \left. \sum_{m=0}^{n-3} k_{m+2,n}\sigma_{m+1,n}D_{m+1}D_{n}^{-1}H_{n+1-j-S_{n-1,m+1}}^{(0,\vec{Q}_{m-2},0)} \right).
        \end{split}
        \end{equation}
    
    Now, suppose $Q_{n-1}=1$. In addition to the contribution $(\mathbb{H}^{\mathbf{i}_n}_i)_0$, there are additional contributions from new paths $X_{i,n+1}$ (as discussed in the type A proof) and from the newly non-intersecting paths $X_{n,n+1},X_{n+1,n+2}$. In particular, the last term in $(\mathbb{H}_j^{\mathbf{i}_n})_{1,n}$, $(\mathbb{H}_j^{\mathbf{i}_n})_{1,n}$ can be multiplied by the weight $X_{n,n+1}$ to account for the new path contributions in the first case and the non-intersecting paths in the second case. Furthermore, the corresponding paths for the 2 cases do not intersect with the Hamiltonians $(\mathbb{H}_j^{\mathbf{i}_n})_{n+2,2n}$ and $(\mathbb{H}_j^{\mathbf{i}_n})_{1,n-1}$, respectively. Explicitly, we obtain the contributions
        \begin{equation}
        \begin{split}
            & (\mathbb{H}_{i-j+S_{n-1,m+1}}^{\mathbf{i}_{n}})_{n+2,2n}X_{n,n+1}\left( \sum_{m=0}^{n-3} k_{m+2,n}\sigma_{m+1,n}D_{m+1}D_{n}^{-1}H_{j-S_{n-1,m+1}}^{(0,\vec{Q}_{m-2},0)} \right) \\
            & \quad = \sum_{m=0}^{n-2}k_{m+2,n}\sigma_{m+1,n}H_{n+1+j-i-S_{n,m+1}}^{\vec{Q}_{n-3}}(\omega_n^{-Q_{n-1}}D_{m+1}D_n)H_{j-S_{n,m+1}}^{\vec{Q}_{m-2}}
        \end{split}
        \end{equation}
    and
        \begin{equation}
        \begin{split}
            & \left( \sum_{m=0}^{n-3} k_{m+2,n}\sigma_{m+1,n}D_{m+1}D_{n}^{-1}H_{n+1+j-i-S_{n-1,m+1}}^{(0,\vec{Q}_{m-2},0)} \right)X_{n,n+1}(\mathbb{H}_j^{\mathbf{i}_n})_{1,n-1} \\
            & \quad = \sum_{m=0}^{n-2}k_{m+2,n}\sigma_{m+1,n}H_{n+1+j-i+S_{n,m+1}}^{\vec{Q}_{n-3}}(\omega_n^{-Q_{n-1}}D_{m+1}D_n)H_{j}^{\vec{Q}_{m-2}}.
        \end{split}
        \end{equation}
    
    Finally, we obtain one last contribution coming from the observation that the paths corresponding to the last terms in $(\mathbb{H}_j^{\mathbf{i}_n})_{1,n}$, $(\mathbb{H}_j^{\mathbf{i}_n})_{1,n}$ can be multiplied together along with $X_{n,n+1}$. This gives us the last contribution:
        \begin{equation}
            \sum_{m,m'=0}^{n-2} k_{m+2,n}k_{m'+2,n}\sigma_{m+1,n}\sigma_{m'+1,n} H_{j+i-n-1+S_{m'+1,m+1}}^{\vec{Q}_{m'-2}}(\omega_n^{-Q_{n-1}}D_{m'+1}D_{m+1}D_n)H^{\vec{Q}_{m-2}}_j
        \end{equation}
        
    By adding up all these contributions, we obtain the equation in \textbf{Theorem 5.5} under the identification $\vec{k}_{n}=(\vec{Q}_{n-1},0)$, proving the theorem. The case of $Q_{n-1}=-1$ is treated analogously to $Q_{n-1}=1$.
\end{proof}

\bibliographystyle{alpha}
\bibliography{mybib}
	
\end{document}